\documentclass[a4paper,12pt]{article}

\usepackage{amsmath,amsfonts,amsthm}
\usepackage{mathrsfs}
\usepackage{bbold}
\usepackage{graphicx}

\newcommand{\1}{{\mathbb{1}}}
\newcommand{\Limsup}{\mathop{\overline{\lim}}\limits}
\newcommand{\argmax}{\mathop{\rm argmax}\limits}
\newcommand{\Ex}{\mathop{\mathbf{\kern 0pt E}}\nolimits}
\newcommand{\Pb}{\mathop{\mathbf{\kern 0pt P}}\nolimits}
\newcommand{\PP}{\mathop{\mathbb{\kern 0pt P}}\nolimits}
\newcommand{\EE}{\mathop{\mathbb{\kern 0pt E}}\nolimits}

\newtheorem{theorem}{Theorem}
\newtheorem{lemma}{Lemma}
\newtheorem{proposition}{Proposition}

\begin{document}

\title{On Hypothesis Testing for Poisson Processes.  Singular Cases}

\author{S. \textsc{Dachian}\\
{\small Universit\'e Blaise Pascal, Clermont-Ferrand, France}\\[8pt]
Yu. A. \textsc{Kutoyants}\\
{\small Universit\'e du Maine, Le Mans, France and}\\
{\small Higher School of Economics, Moscow, Russia}\\[8pt]
L. \textsc{Yang}\footnote{Corresponding author. E-mail address:
  Lin\_Yang.Etu@univ-lemans.fr}\\
{\small Universit\'e du Maine, Le Mans, France}}

\date{}

\maketitle

\begin{abstract}
We consider the problem of hypothesis testing in the situation where the first
hypothesis is simple and the second one is local one-sided composite. We
describe the choice of the thresholds and the power functions of different
tests when the intensity function of the observed inhomogeneous Poisson
process has two different types of singularity: cusp and discontinuity. The
asymptotic results are illustrated by numerical simulations.
\end{abstract}


\medskip


\section{Introduction}

This is the second part of the study devoted to hypothesis testing problems in
the case when the observations are inhomogeneous Poisson processes. The first
part was concerned with the regular (smooth) case \cite{DKYR}, while this
second part deals with non regular (singular) situations.  We suppose that the
intensity function of the observed inhomogeneous Poisson process depends on
the unknown parameter $\vartheta $ in a non regular way (for example, the
Fisher information is infinite).  The basic hypothesis is always simple
($\vartheta =\vartheta _1$) and the alternative is one-sided composite
($\vartheta >\vartheta _1$).  In the first part we studied the asymptotic
behavior of the Score Function test (SFT), of the General Likelihood Ratio
test (GLRT), of the Wald test (WT) and of two Bayes tests (BT1 and BT2). It
was shown that the tests SFT, GLRT and WT are locally asymptotically uniformly
most powerful.  In the present work we study the asymptotic behavior of the
GLRT, WT, BT1 and BT2 in two non regular situations. More precisely, we study
the tests when the intensity functions has a cusp-type singularity or a
jump-type singularity. In both cases the Fisher information is infinite. The
local alternatives are obtained by the reparametrization $\vartheta=\vartheta
_1+u\varphi _n$, $u> 0$. The rate of convergence $\varphi _n\rightarrow 0$
depends on the type of singularity. In the cusp case $\varphi _n\sim
n^{-\frac{1}{2\kappa +1}}$, where $\kappa$ is the order of the cusp, and in
the discontinuous case $\varphi _n\sim n^{-1}$. Our goal is to describe the
choice of the thresholds and the behavior of the power functions as
$n\rightarrow \infty $. The important difference between regular and singular
cases is the absence of the criteria of optimality. This leads to a situation
when the comparison of the power functions can be only done numerically. That
is why we present the results of numerical simulations of the limit power
functions and the comparison of them with the power functions with small and
large volumes of observations (small and large $n$).

Recall that $X=\left(X_t,\ t\geq 0 \right)$ is an inhomogeneous Poisson
process with intensity function $\lambda \left(t\right)$, $t\geq 0$, if
$X_0=0$ and the increments of $X$ on disjoint intervals are independent and
distributed according to the Poisson law
$$
\Pb\left\{X_t-X_s=k\right\}=\frac{\left(\int_s^t \lambda
\left(t\right) {\rm d}t\right)^k }{k!}\exp\left\{-\int_s^t \lambda
\left(t\right) {\rm d}t\right\}.
$$
We suppose that the intensity function depends on some one-dimensional
parameter, that is, $\lambda \left(t\right)=\lambda \left(\vartheta
,t\right)$.  The basic hypothesis is simple: $\vartheta =\vartheta _1$, while
the alternative is one-sided composite: $\vartheta >\vartheta _1$.

The hypothesis testing problems for inhomogeneous Poisson processes were
studied by many authors (see, for example, \cite{LW87}, \cite{FP11},
\cite{DKYR} and the references therein).

\section{Preliminaries}

We consider the model of $n$ independent observations of an inhomogene\-ous
Poisson process: $X^n=\left(X_1,\ldots,X_n\right)$, where
$X_j=\left(X_j\left(t\right), 0\leq t\leq \tau \right)$, $j=1,\ldots,n$, are
Poisson processes with
$$
\Ex_\vartheta X_j\left(t\right)=\Lambda \left(\vartheta
,t\right)=\int_{0}^{t}\lambda \left(\vartheta ,s\right)\;{\rm d}s.
$$
We use here the same notations as in \cite{DKYR}.  In particular, $\vartheta $
is one-dimensional parameter and $\Ex_\vartheta $ is the mathematical
expectation in the case when the true value is $\vartheta$.  The intensity
function is supposed to be separated from zero on $\left[0,\tau \right]$. The
measures corresponding to Poisson processes with different values of
$\vartheta $ are equivalent and the likelihood function is defined by the
equality
$$
L(\vartheta ,X^n)=\exp\left\{\sum_{j=1}^{n}\int_{0}^{\tau }\ln\lambda
\left(\vartheta ,t\right){\rm d}X_j\left(t\right)-n\int_{0}^{\tau
}\left[\lambda \left(\vartheta ,t\right)-1 \right]{\rm d}t\right\}.
$$

In non-regular situations we do not have a LAUMP test, and it is interesting
to compare the power functions of different tests with the power function of
the Neyman-Pearson test (N-PT). Let us recall the definition of the N-PT.
Suppose that we have two simple hypotheses $\mathscr{ H}_1:\vartheta
=\vartheta_1 $ and $\mathscr{ H}_2:\vartheta =\vartheta_2$ and our goal is to
construct a test $\bar\psi _n\left(X^n\right)$ of size $\varepsilon $, that
is, a test with given probability of the error of the first kind
$\Ex_{\vartheta_1}\bar\psi _n\left(X^n\right)=\varepsilon $. As usually, the
test $\bar\psi _n\left(X^n\right)$ is the probability to reject the hypothesis
$\mathscr{ H}_1 $ and, of course, to accept the hypothesis $\mathscr{ H}_2$.

Let us denote the likelihood ratio statistic as
$$
L\left(\vartheta_2 ,\vartheta
_1,X^n\right)=L\left(\vartheta_2,X^n\right)/L\left(\vartheta_1,X^n\right).
$$
Then, by the Neyman-Pearson Lemma \cite{LR}, the N-PT is
$$
\psi_n^*\left(X^n\right)=\begin{cases}
1,& \, {\rm if}\,\, \quad L\left(\vartheta_2 ,\vartheta
_1,X^n\right)>d_\varepsilon ,\\
q_\varepsilon, & \,{\rm if}\, \quad L\left(\vartheta_2 ,\vartheta
_1,X^n\right)=d_\varepsilon ,\\
0,& \, {\rm if} \,\quad L\left(\vartheta_2 ,\vartheta _1,X^n\right)
<d_\varepsilon,
\end{cases}
$$
where the constants $d_\varepsilon $ and $q_\varepsilon $ are solutions of the
equation
$$
\Pb_{\vartheta _1}\left(L\left(\vartheta_2 ,\vartheta
_1,X^n\right)>d_\varepsilon \right)+q_\varepsilon \Pb_{\vartheta
  _1}\left(L\left(\vartheta_2 ,\vartheta _1,X^n\right)=d_\varepsilon \right)
=\varepsilon.
$$

In this work we consider the construction of the tests in the following
hypothesis testing problem
\begin{equation}
\label{ht}
\begin{aligned}
\mathscr{ H}_1\quad &:\qquad \vartheta =\vartheta _1,\\
\mathscr{ H}_2\quad &:\qquad \vartheta >\vartheta _1,
\end{aligned}
\end{equation}
that is, we have a simple hypothesis against one-sided composite alternative.

The log likelihood ratio function can be written as follows:
$$
\ln L(\vartheta ,\vartheta
_1,X^n)=\sum_{j=1}^{n}\int_{0}^{\tau }\ln\frac{\lambda
\left(\vartheta ,t\right)}{\lambda \left(\vartheta_1
  ,t\right)}{\rm d}X_j\left(t\right)-n\int_{0}^{\tau }\left[\lambda
  \left(\vartheta ,t\right)-\lambda \left(\vartheta_1
  ,t\right) \right]{\rm d}t.
$$

The power function of a test $\bar \psi_n\left(X^n\right)$ is $\beta
\left(\bar \psi _n,\vartheta\right) =\Ex_\vartheta \bar
\psi_n\left(X^n\right)$, $\vartheta >\vartheta _1$.

We denote $\mathcal{ K}_\varepsilon$ the class of tests $\bar \psi _n$ of
asymptotic size $\varepsilon$:
$$
\mathcal{ K}_\varepsilon =\left\{\bar \psi _n\quad :\quad \lim_{n\rightarrow
  \infty }\Ex_{\vartheta _1}\bar\psi _n\left(X^n\right)=\varepsilon \right\}.
$$
In this work we study several tests which belong to the class $\mathcal{
  K}_\varepsilon $. To compare these tests by their power functions we
consider, as usual, the approach of \textit{close} or \textit{contiguous}
alternatives (since for any fixed alternative the power functions of all tests
converge to the same value $1$). We put $\vartheta =\vartheta _1+ \varphi
_nu$, where $\varphi _n=\varphi _n\left(\vartheta _1\right)>0$. Here $\varphi
_n\rightarrow 0$ and the rate of convergence depends on the type of
singularity of the intensity function.

Now the initial problem of hypothesis testing can be rewritten as follows:
\begin{equation}
\label{ht-u}
\begin{aligned}
\mathscr{ H}_1\quad &:\qquad u =0,\\
\mathscr{ H}_2\quad &:\qquad u>0.
\end{aligned}
\end{equation}

The considered tests are usually of the form
$$
\bar\psi _n=\1_{\left\{Y_n\left(X^n\right)>c_\varepsilon
  \right\}}+q_\varepsilon \1_{\left\{Y_n\left(X^n\right)=c_\varepsilon
  \right\}},
$$
where the constant $c_\varepsilon $ is defined with the help of the limit
random variable $Y$ (suppose that $Y_n\Longrightarrow Y$ under hypothesis
$\mathscr{H}_1$) by
$$
\Ex_{\vartheta _1}\bar\psi _n=\Pb_{\vartheta
  _1}\left\{Y_n\left(X^n\right)>c_\varepsilon \right\}+ q_\varepsilon
\Pb_{\vartheta _1}\left\{Y_n\left(X^n\right)=c_\varepsilon
\right\}\longrightarrow \Pb_{\vartheta _1}\left\{Y>c_\varepsilon
\right\}=\varepsilon
$$
if the limit random variable $Y$ is continuous, and by
$$
\Pb_{\vartheta _1}\left\{Y>c_\varepsilon \right\}+ q_\varepsilon
\Pb_{\vartheta _1}\left\{Y=c_\varepsilon \right\}=\varepsilon
$$
if $Y$ has distribution function with jumps.

The corresponding power function will be denoted
$$
\beta\left(\bar\psi _n,u\right)=\Ex_{\vartheta_1+\varphi _n u}\,\bar\psi
_n,\qquad u>0,
$$
and the comparison of the tests will be carried in terms of their limit power
functions.

We consider two different non regular models. In both of them, the intensity
function $\lambda \left(\vartheta,t \right)$ is not differentiable and the
Fisher information is infinite. More precisely, we study the behavior of the
tests in two situations.  The first one is when the intensity function has a
cusp-type singularity (it is continuous but not differentiable), and the
second one is when it has a jump-type singularity (it is discontinuous). In
both cases the intensity functions $\lambda \left(\vartheta ,t\right)$ has no
derivative at the point $t=\vartheta $.

Note that these statistical models were already studied before in the problems
of parameter estimation (see~\cite{Da03} for the cusp-type singularity
and~\cite{Kut98} for discontinuous intensity function), so here we concentrate
on the properties of the tests. The main tool is, of course, the limit
behavior of the normalized likelihood ratio function, which was already
established before in the mentioned works but in a slightly different
settings. The proofs given in this work are mainly based on the results
presented in \cite{Da03} and \cite{Kut98}.

 Recall that in the non regular cases considered in this work we do not have a
 LAUMP tests, and that is why a special attention is paid to numerical
 simulations of the limit power functions.

\section{Cusp-type singularity}

Suppose that the intensity function of the observed Poison processes is
$$
\lambda \left(\vartheta ,t\right)=a\left|t-\vartheta \right|^\kappa +
h\left(t\right),\qquad 0\leq t\leq \tau ,\qquad \vartheta \in \Theta
=[\vartheta _1,b),
$$
where $\kappa \in \left(0,1/2\right)$, $\vartheta _1 >0$, $b\leq\tau $, $a\ne
0$, and $h\left(\cdot \right)$ is a known positive bounded function.

To study the local alternatives we introduce the normalizing function
$$
\varphi _n=n^{-\frac{1}{2H }}\Gamma _{\vartheta _1}^{-\frac{1}{H}},\qquad
\Gamma_{\vartheta_1} ^2=\frac{2a^2B\left(\kappa +1,\kappa
  +1\right)}{h\left({\vartheta_1} \right)} \left[\frac{1}{\cos\left(\pi \kappa
    \right)}-1\right],
$$
where $B\left(\cdot ,\cdot \right)$ is the \textit{Beta-function} and
$H=\kappa +\frac{1}{2}$ is the \textit{Hurst parameter}. As usually, the
change of variables $\vartheta =\vartheta _1+\varphi _n{u}$ reduces the
initial hypothesis testing problem \eqref{ht} to the problem \eqref{ht-u}.

We introduce the stochastic process
$$
Z\left(u\right)=\exp\left\{W^H\left(u\right)-\frac{\left|u\right|^{2H}}{2}
\right\},\qquad u\in\mathbb{R},
$$
where $W^H\left(\cdot \right)$ is a {\it fractional Brownian motion}. Further,
we define the random variable $\hat u$ by the relation
$$
Z\left(\hat u\right)=\sup_{v\geq 0}Z\left(u\right),
$$
and we introduce $h_\varepsilon$ and $g_\varepsilon $ as the solutions of the
equations
\begin{equation}
\label{thresh}
\Pb\left(Z\left(\hat u\right)>h_\varepsilon \right)=\varepsilon\quad
\text{and}\quad \Pb\left(\hat u>g_\varepsilon \right)=\varepsilon
\end{equation}
respectively.

Note that $(Z\left(u\right),\ u\geq 0)$ is the likelihood ratio of a similar
hypothesis testing problem ($u=0$ against $u>0$) in the case of observations
$\left(Y\left(t\right),\ t\geq 0\right)$ of the following type
$$
{\rm d}Y\left(t\right)=\1_{\left\{t<u\right\}}\;{\rm d}t+{\rm
  d}W^H\left(t\right),\qquad t\geq 0.
$$
The uniformly most powerful test in this problem does not exist, and we do not
have a LAUMP tests in our problem.

\subsection{GLRT}

The GLRT is defined by the relations
$$
\hat\psi_n\left(X^n\right)=\1_{\left\{Q\left(X^n\right)> h_\varepsilon
  \right\}},
$$
where $h_\varepsilon$ is the solution of the first of the
equations~\eqref{thresh},
$$
Q\left(X^n\right)= \sup_{\vartheta >\vartheta _1}{L\left(\vartheta,\vartheta
  _1, X^n\right)}=L\left(\hat\vartheta_n,\vartheta _1 ,X^n\right),
$$
and $\hat\vartheta _n$ is the maximum likelihood estimator (MLE).

Let us introduce the function
$$
\hat\beta \left(u_*\right)=\Pb \left\{\sup_{u> 0} \left[W^H\left(u
  \right)-\frac{\left|u-u_*\right|^{2H}}{2}\right] >\ln h_\varepsilon
-\frac{\left|u_*\right|^{2H}}{2}\right\},\quad u_*\geq 0 .
$$

The properties of the GLRT are given in the following Proposition.

\begin{proposition}
The GLRT $\hat\psi_n\left(X^n\right)$ belongs to $\mathcal{ K}_\varepsilon $
and its power function in the case of local alternatives $\vartheta =\vartheta
_1+\varphi _nu_*$, $u_*>0$, has the following limit:
$$
\beta \left(\hat\psi _n,u_*\right)\longrightarrow \hat\beta \left(u_*\right).
$$
\end{proposition}

\begin{proof}
Introduce the normalized likelihood ratio process
$$
Z_n\left(u\right)=L\left(\vartheta _1+\varphi _n
u,X^n\right)=\frac{L\left(\vartheta _1+\varphi _n
  u,X^n\right)}{L\left(\vartheta _1,X^n\right) },\ \:\ u\in
\mathbb{U}_n^+=\bigl[0,\varphi _n^{-1} \left(b-\vartheta _1\right)\bigr),
$$
and let the function $Z_n\left(u\right)$ be linearly decreasing to zero on the
interval $\left[\varphi _n^{-1} \left(b-\vartheta _1\right),\varphi _n^{-1}
  \left(b-\vartheta _1\right)+1 \right]$ and equal to $0$ for all $u>\varphi
_n^{-1} \left(b-\vartheta _1\right)+1 $. Now the random function
$Z_n\left(\cdot\right)$ is defined on $\mathbb{R}_+$.

Let us fix some $d\leq 0$ and denote $\mathscr{ C}_d= \mathscr{
  C}_d\left(\mathbb{R}_d\right)$ the space of continuous functions on
$\mathbb{R}_d=[d,\infty )$ with the property $\lim_{v\rightarrow \infty }
  z\left(v\right)=0$. Introduce the uniform metric on this space and denote
  $\mathcal{ B}$ the corresponding Borel $\sigma $-algebra.

When we study the likelihood ratio process under hypothesis $\mathscr{ H}_1$,
we take $d=0$ and consider the corresponding measurable space $\left(\mathscr{
  C}_0, \mathcal{ B} \right)$.  Under the alternative $\vartheta =\vartheta
_{u_*}=\vartheta _1+\varphi _nu_*$, $u_*>0$, we will use this space with $
d=-u_*$.

Let ${\bf Q}$ be the measure induced on the measurable space $\left(\mathscr{
  C}_0, \mathcal{ B} \right)$ by the stochastic processes
$(Z\left(u\right),\ u\geq 0)$, and ${\bf Q}_n^{(\vartheta)}$ be the measure
induced (under the true value $\vartheta$) on the same space by the processes
$(Z_n\left(u\right),\ u\geq 0)$. The continuity with probability $1$ of the
random functions $(Z_n\left(u\right),\ u\geq 0)$ follows from the
inequality~\eqref{cont} below and the Kolmogorov theorem.

Suppose that we already proved the following weak convergence
\begin{equation}
\label{wc}
{\bf Q}_n^{(\vartheta_1)}\Longrightarrow {\bf Q}.
\end{equation}
Then the distribution of any continuous in the uniform metric functional $\Phi
\left(Z_n\right)$ converge to the distribution of $\Phi \left(Z\right)$.  In
particular, if we take
$$
\Phi \left(z\right)=\sup_{u\geq 0}z\left(u\right)-h_\varepsilon ,
$$
we obtain
\begin{align*}
\Pb_{\vartheta _1}\left\{\sup_{\vartheta >\vartheta
  _1}{L\left(\vartheta,\vartheta _1, X^n\right)}>h_\varepsilon \right\}&=
\Pb_{\vartheta _1}\left\{\sup_{u>0} Z_n\left(u\right)>h_\varepsilon
\right\}\\
\longrightarrow{}&\Pb\left\{\sup_{u> 0} Z\left(u\right)>h_\varepsilon
\right\}=\Pb \left\{ Z\left(\hat u\right) > h_\varepsilon
\right\}=\varepsilon.
\end{align*}
Therefore the test $\hat\psi _n\in \mathcal{ K}_\varepsilon $.

Let us note, that we do not know an analytic solution of the equation defining
the constant $h_\varepsilon $, that is why below we turn to numerical
simulations (see Section \ref{SSC}).  Note also that $h_\varepsilon
=h_\varepsilon \left(H\right)$ and does not depend on $\Gamma _{\vartheta
  _1}$.

To study the power function we consider the same likelihood ratio process but
under the alternative $\vartheta _{u_*}=\vartheta _1+\varphi _nu_*$. We can
write
\begin{align*}
Z_n\left(u\right)&=\frac{L\left(\vartheta _1+ \varphi_n
  u,X^n\right)}{L\left(\vartheta _1,X^n\right)}= \frac{L\left(\vartheta
  _{u_*},X^n\right)}{L\left(\vartheta _1,X^n\right)}
\frac{L\left(\vartheta _1+ \varphi_n u,X^n\right)}{L\left(\vartheta
  _{u_*},X^n\right)} \\
&=\left(\frac{L\left(\vartheta _{u_*}-\varphi _n
  u_*,X^n\right)}{L\left(\vartheta _{u_*},X^n\right)} \right)^{-1}
\frac{L\left(\vartheta _{u_*}+ (u-u_*)\varphi_n,X^n\right)}{L\left(\vartheta
  _{u_*},X^n\right)}\\
&=\tilde Z_n\left(-u_*\right)^{-1} \tilde Z_n\left(u-u_*\right)
\end{align*}
with an obvious notation. The difference between $Z_n\left(\cdot \right)$ and
$\tilde Z_n\left(\cdot \right)$ is that the ``reference value'' in the first
case is fixed (is equal to $\vartheta _1$) and in the second case it is
``moving'' (is equal to $\vartheta_{u_*}=\vartheta _1+\varphi _n u_*$). The
random variable $\tilde Z_n\left(-u_*\right)$ converge in distribution to $
Z\left(-u_*\right)$. For the stochastic process $(\tilde
Z_n\left(u-u_*\right),\ u\geq 0)$ we have a similar convergence, and so, for
any fixed $u\geq 0$, we have
$$
\left(\tilde Z_n\left(-u_*\right),\tilde
Z_n\left(u-u_*\right)\right)\Longrightarrow \left(Z\left(-u_*\right),
Z\left(u-u_*\right)\right).
$$

Now, let $\tilde {\bf Q}$ be the measure induced on the measurable space
$\left(\mathscr{ C}_{-u_*}, \mathcal{ B} \right)$ by the stochastic processes
$(Z\left(u\right),\ u\geq -u_*)$, and $\tilde{\bf Q}_n^{(\vartheta)}$ be the
measure induced (under the true value $\vartheta$) on the same space by the
stochastic processes $(\tilde Z_n\left(u\right),\ u\geq -u_*)$. Suppose that
we already proved the weak convergence
\begin{equation}
\label{wcon}
\tilde{\bf Q}_n^{(\vartheta_{u_*})}\Longrightarrow \tilde{\bf Q}.
\end{equation}
Then for the power function we can write
\begin{align*}
&\Pb_{\vartheta _{u_*}}\left\{\sup_{u> 0} Z_n\left(u\right)>h_\varepsilon
  \right\}\\
&\quad =\Pb_{\vartheta _{u_*}}\left\{\tilde Z_n\left(-u_*\right)^{-1}\sup_{u>
    0}\frac{L\left(\vartheta _{u_*}+
    (u-u_*)\varphi_n,X^n\right)}{L\left(\vartheta _{u_*},X^n\right)}
  >h_\varepsilon \right\}\\
&\quad\longrightarrow \Pb \left\{Z\left(-u_*\right)^{-1}\sup_{u> 0}\ \exp
  \left\{W^H\left(u-u_*\right)
  -\frac{\left|u-u_*\right|^{2H}}{2}\right\}>h_\varepsilon \right\}\\
&\quad=\Pb \left\{\sup_{u> 0}
  \left[-W^H\left(-u_*\right)+W^H\left(u-u_*\right) -\frac{\left|u-u_*
      \right|^{2H}}{2}+\frac{\left|u_*\right|^{2H}}{2}\right] >\ln
  h_\varepsilon \!\right\}\\
&\quad=\Pb \left\{\sup_{u> 0} \left[W^H\left(u\right)-\frac{\left|u-u_*
      \right|^{2H}}{2}\right] >\ln h_\varepsilon
  -\frac{\left|u_*\right|^{2H}}{2}\right\}= \hat\beta \left(u_*\right) .
\end{align*}

This limit power function is obtained below with the help of numerical
simulations (see Section \ref{SSC}).

Let us also note that the limit (under the alternative $\vartheta
_{u_*}=\vartheta _1+\varphi _n u_*$) of the likelihood ratio process
$\left(Z_n\left(u\right),\ u\geq 0\right)$ is the process
$\left(Z\left(u,u_*\right),\ u\geq 0\right)$ defined by
$$
Z(u,u_*)=W^H\left(u\right)-\frac{\left|u-u_*
  \right|^{2H}}{2}+\frac{\left|u_*\right|^{2H}}{2}.
$$

To finish the proof we need to verify the convergence \eqref{wcon}. To do this
we follow the proof of the convergence \eqref{wc} given in \cite{Da03}.  We
introduce the following relations.

\begin{enumerate}
\item\textit{The finite-dimensional distributions of $(\tilde
  Z_n\left(u\right),\ u\geq -u_*)$ converge to those of
  $\left(Z\left(u\right),\ u\geq -u_*\right)$.}
\item\textit{There exists a positive constant $C$ such that}
\begin{equation}
\label{cont}
\Ex_{\vartheta_{u_*}}\left|\tilde Z_n^{1/2}\left(u_2\right)-\tilde
Z_n^{1/2}\left(u_1\right)\right|^2\leq C\,\left|u_2-u_1\right| ^{2H},\qquad
u_1,u_2\geq -u_*.
\end{equation}
\item\textit{There exists a positive constant $c$ such that}
\begin{equation}
\label{gd}
\Ex_{\vartheta_{u_*}}\tilde Z_n^{1/2}\left(u\right)\leq
\exp\left\{-c\left|u-u_*\right|^{2H}\right\},\qquad u\geq -u_*.
\end{equation}
\end{enumerate}

The proofs of these relations are slight modifications of the proofs given
in~\cite{Da03}. Note that the characteristic function of the vector
$$
\tilde Z_n\left(u_1\right),\ldots,\tilde Z_n\left(u_k\right)
$$
can be written explicitly and the convergence of this characteristic function
to the corresponding limit characteristic function can be checked directly
(see Lemma~5 of~\cite{Da03}). The inequalities~\eqref{cont} and~\eqref{gd}
follow from the Lemma~6 and Lemma~7 of~\cite{Da03} respectively.

These relations allow us to obtain the weak convergence~\eqref{wcon} by
applying the Theorem~1.10.1 of~\cite{IH81}. Note that the
convergence~\eqref{wc} is a particular case of~\eqref{wcon} with $u_*=0$.
\end{proof}

\subsection{Wald test}

Recall that the MLE $\hat\vartheta _n$ is defined by the equation
$$
L\left(\hat\vartheta _n,\vartheta _1,X^n\right)=\sup_{\vartheta \in\Theta
}L\left(\vartheta ,\vartheta _1,X^n\right) .
$$
The Wald test (WT) has the following form:
$$
\psi_n^\circ\left(X^n\right)=\1_{\left\{\varphi_n^{-1}\left(\hat\vartheta
  _n-\vartheta _1\right) >g_\varepsilon \right\}},
$$
where $g_\varepsilon$ is the solution of the second of the
equations~\eqref{thresh}.

Introduce as well the random variable $\hat u_*$ as solution of the equation
$$
Z\left(\hat u_*\right)=\sup_{u\geq -u_*}Z\left(u\right).
$$

\begin{proposition}
The WT $\psi_n^\circ\left(X^n\right)$ belongs to $\mathcal{ K}_\varepsilon $
and its power function in the case of local alternatives $\vartheta =\vartheta
_1+\varphi _n u_*$, $u_*>0$, has the following limit:
$$
\beta \left(\psi _n^\circ,u_*\right)\longrightarrow \beta^\circ
\left(u_*\right)=\Pb\left(\hat u_*>g_\varepsilon -u_*\right).
$$
\end{proposition}

\begin{proof}
The MLE (under hypothesis $\mathscr{ H}_1$) converges in distribution
$$
\varphi_n^{-1}\left(\hat\vartheta _n-\vartheta _1\right)\Longrightarrow \hat
u.
$$
Hence $\psi_n^\circ\in \mathcal{ K}_\varepsilon$. For the proof
see~\cite{Da03}. Recall that this convergence is a consequence of the weak
convergence~\eqref{wc}.

Let us study this estimator under the alternative $\vartheta_{u_*}=\vartheta
_1+\varphi _n u_*$, $u_*>0$. We have
\begin{align*}
&\Pb_{\vartheta_{u_*}}\left(\varphi_n^{-1}\left(\hat\vartheta _n-\vartheta
  _{u_*}\right)<x \right)\\
&\:\ =\Pb_{\vartheta _{u_*}}\left(\sup_{\varphi_n^{-1}\left( \theta
    -\vartheta_{u_*}\right)<x }L\left(\theta ,\vartheta_{u_*},X^n\right) >
  \sup_{\varphi_n^{-1}\left( \theta -\vartheta_{u_*}\right) \geq x
  }L\left(\theta ,\vartheta_{u_*},X^n\right) \right)\\
&\:\ =\Pb_{\vartheta_{u_*}}\left(\sup_{-u_*\leq u<x }\tilde Z_n\left(u\right)
  > \sup_{u\geq x }\tilde Z_n\left(u\right)
  \right)\longrightarrow\Pb\left(\sup_{-u_*\leq u<x }Z\left(u\right) >
  \sup_{u\geq x }Z\left(u\right) \right) \\
&\:\ =\Pb\left(\hat u_* <x\right).
\end{align*}
Here, as before,
$$
\tilde Z_n\left(u\right)=\frac{L\left(\vartheta _{u_*}+\varphi _n
  u,X^n\right)}{L\left(\vartheta _{u_*},X^n\right)},\qquad u\geq -u_*.
$$

Now, the limit of the power function of the WT is deduced from this
convergence:
\begin{align*}
\beta \left(\psi _n^\circ,u_*\right)&=\Pb_{\vartheta _{u_*}}
\left\{\varphi_n^{-1}\left(\hat\vartheta_n-\vartheta
_{u_*}\right)+u_*>g_\varepsilon \right\}\\
&\longrightarrow \Pb\left\{\hat u_* >g_\varepsilon-u_*\right\}=\beta^\circ
\left(u_*\right),
\end{align*}
which concludes the proof.
\end{proof}

Let us note, that we can also give another representation of the limit power
function using the process $\left(Z\left(u,u_*\right),\ u\geq 0\right)$:
$$
\beta \left(\psi _n^\circ,u_*\right)\longrightarrow\Pb\left\{\hat u^*
>g_\varepsilon\right\}=\beta^\circ \left(u_*\right),
$$
where $\hat u^*$ is solution of the equation
$$
Z\left(\hat u^*\right)=\sup_{u\geq 0}Z\left(u,u_*\right).
$$

The threshold $g_\varepsilon $ and the power function $\beta^\circ
\left(\cdot\right)$ are obtained below by numerical simulations (see Section
\ref{SSC}).

\subsection{Bayes tests}

Suppose that the parameter $\vartheta $ is a random variable with \textit{a
  priori} density $p\left(\theta \right)$, $\vartheta _1\leq \theta <b$. This
function is supposed to be continuous and positive.

We consider two Bayes tests. The first one is based on the Bayes estimator,
while the second one is based on the averaged likelihood ratio.

\bigskip
\bigskip

The first test, which we call BT1, is similar to WT, but is based on the Bayes
estimator (BE) rather than on the MLE. Suppose that the loss function is
quadratic. Then the BE $\tilde\vartheta _n$ is given by the following
conditional expectation:
$$
\tilde\vartheta _n=\int_{\vartheta _1}^{b}\theta p\left(\theta|X^n \right){\rm
  d}\theta= \frac{\int_{\vartheta _1}^{b}\theta p\left(\theta
  \right)L\left(\theta ,X^n\right){\rm d}\theta }{\int_{\vartheta _1}^{b}
  p\left(\theta \right)L\left(\theta ,X^n\right){\rm d}\theta }.
$$

We introduce the test BT1 as
$$
\tilde\psi_n\left(X^n\right)=\1_{\left\{\varphi_n^{-1}\left(\tilde\vartheta
  _n-\vartheta _1\right) >k_\varepsilon \right\}},
$$
where the constant $k_\varepsilon $ is solution of the equation
$$
\Pb\left(\tilde u>k_\varepsilon \right)=\varepsilon ,\qquad \tilde
u=\frac{\int_{0}^{\infty }vZ\left(v\right){\rm d}v}{\int_{0}^{\infty
  }Z\left(v\right){\rm d}v}.
$$

Introduce as well the function
$$
\tilde\beta \left(u_*\right)=\Pb\left(\tilde u_*>k_\varepsilon
-u_*\right),\qquad \tilde u_*=\frac{\int_{-u_*}^{\infty }vZ\left(v\right){\rm
    d}v}{\int_{-u_*}^{\infty }Z\left(v\right){\rm d}v},\qquad u_*\geq 0.
$$

\begin{proposition}
The BT1 $\tilde\psi_n\left(X^n\right)$ belongs to $\mathcal{ K}_\varepsilon $
and its power function in the case of local alternatives $\vartheta =\vartheta
_1+\varphi _n u_*$, $u_*>0$, has the following limit:
$$
\beta\left(\tilde\psi_n,u_*\right)\longrightarrow \tilde\beta
\left(u_*\right).
$$
\end{proposition}

\begin{proof}
The Bayes estimator $\tilde \vartheta _n$ is consistent and has the following
limit distribution (under hypothesis $\mathscr{ H}_1$)
$$
\varphi _n^{-1} \left(\tilde\vartheta _n-\vartheta _1\right)\Longrightarrow
\tilde u
$$
(for the proof see \cite{Da03}). Hence $\tilde\psi_n\left(X^n\right)\in
\mathcal{ K}_\varepsilon $.

For the power function we have
\begin{align*}
\beta \left(\tilde\psi _n,u_*\right)&=\Pb_{\vartheta
  _{u_*}}\left\{\varphi_n^{-1}\left(\tilde\vartheta _n-\vartheta _1\right)
>k_\varepsilon \right\}\\
&=\Pb_{\vartheta _{u_*}}\left\{\varphi_n^{-1}\left(\tilde\vartheta
_n-\vartheta _{u_*}\right) >k_\varepsilon -u_*\right\}.
\end{align*}
Let us study the normalized difference $\tilde
u_n=\varphi_n^{-1}\left(\tilde\vartheta _n-\vartheta _{u_*}\right) $. We can
write (using the change of variables $\theta =\vartheta_{u_*}+\varphi _nv$)
\begin{align*}
&\int_{\vartheta _1}^{b}\theta p\left(\theta
  \right)L\left(\theta,\vartheta_{u_*} ,X^n\right){\rm d}\theta\\
&\qquad = \varphi _n\int_{-u_*}^{\varphi _n^{-1}\left(b-\vartheta
    _{u_*}\right)}\left(\vartheta _{u_*}+\varphi _n v\right) p\left(\vartheta
  _{u_*} +\varphi _n v\right)L\left(\vartheta_{u_*}+\varphi _n v,\vartheta
  _{u_*} ,X^n\right){\rm d}v\\
&\qquad = \varphi _n\int_{-u_*}^{\varphi
    _n^{-1}\left(b-\vartheta_{u_*}\right)}\left(\vartheta_{u_*}+\varphi_n
  v\right) p\left(\vartheta_{u_*} +\varphi_n v\right)\tilde
  Z_n\left(v\right){\rm d}v.
\end{align*}
Hence
$$
\tilde u_n=\frac{\int_{-u_*}^{\varphi _n^{-1}\left(b-\vartheta_{u_*}\right)}v
  p\left(\vartheta_{u_*}+\varphi _n v\right)\tilde Z_n\left(v\right){\rm d}v
}{\int_{-u_*}^{\varphi _n^{-1}\left(b-\vartheta_{u_*}\right)}
  p\left(\vartheta_{u_*} +\varphi _n v\right)\tilde Z_n\left(v\right){\rm
    d}v}\Longrightarrow \frac{\int_{-u_*}^{\infty }v Z\left(v\right){\rm d}v
}{\int_{-u_*}^{\infty } Z\left(v\right){\rm d}v}=\tilde u_*
$$
(since $p\left(\vartheta_{u_*}+\varphi_n v\right)\longrightarrow
p\left(\vartheta _1\right)>0 $ and $\tilde Z_n\Longrightarrow Z$). The
detailed proof is based on the properties 1--3 of the likelihood ratio (see
\cite{Da03} or \cite[Theorem 1.10.2]{IH81}).
\end{proof}

Let us note, that we can also give another representation of the limit power
function using the process $\left(Z\left(u,u_*\right),\ u\geq 0\right)$:
$$
\beta \left(\tilde\psi_n,u_*\right)\longrightarrow\Pb\left(\tilde
u^*>k_\varepsilon\right)=\tilde\beta \left(u_*\right),
$$
where $\displaystyle\tilde u^*=\frac{\int_{0}^{\infty }v
  Z\left(v,u_*\right){\rm d}v }{\int_{0}^{\infty } Z\left(v,u_*\right){\rm
    d}v}$.

\bigskip
\bigskip

The second test, which we call BT2, is given by
$$
\tilde\psi_n^\star\left(X^n\right)=\1_{\left\{R_n\left(X^n\right)>m_\varepsilon
  \right\}},\qquad R_n\left(X^n\right)=\frac{ \tilde
  L_n\left(X^n\right)}{p\left(\vartheta _1\right)\varphi _n}.
$$
Here
$$
\tilde L_n\left(X^n\right)=\int_{\vartheta _1}^{b}L\left(\theta ,\vartheta
_1,X^n\right)p\left(\theta \right){\rm d}\theta,
$$
and $m_\varepsilon $ is solution of the equation
$$
\Pb\left\{ \int_{0}^{\infty } Z\left(v\right){\rm d}v >m_\varepsilon
\right\}=\varepsilon.
$$

Introduce as well the function
$$
\tilde\beta^\star\left(u_*\right)=\Pb\left(Z\left(-u_*\right)^{-1}\int_{-u_*}^{\infty
}Z\left(v\right)\,{\rm d}v >m_\varepsilon \right).
$$

\begin{proposition}
The BT2 $\tilde\psi_n^\star\left(X^n\right)$ belongs to $\mathcal{
  K}_\varepsilon $ and its power function in the case of local alternatives
$\vartheta =\vartheta _1+\varphi _n u_*$, $u_*>0$, has the following limit:
$$
\beta\left(\tilde\psi_n^\star,u_*\right)\longrightarrow
\tilde\beta^\star\left(u_*\right).
$$
\end{proposition}

\begin{proof}
Let us first recall how this test was obtained. Introduce the mean error of
the second kind $\bar\alpha \left(\bar\psi _n\right)$ under alternative
$\mathscr{ H}_2$ of an arbitrary test $\bar\psi _n$ as
$$
\bar\alpha \left(\bar\psi _n\right)=\int_{\vartheta _1}^{b}\Ex_\theta \bar\psi
_n\left(X^n\right)\,p\left(\theta \right)\,{\rm d}\theta=\EE\bar\psi _n,
$$
where $\EE$ is the double mathematical expectation, that is, the expectation
with respect to the measure
$$
\PP \left(X^n\in A\right)=\int_{\vartheta _1}^{b}\Pb_\theta \left(X^n\in
A\right)p\left(\theta \right)\,{\rm d}\theta .
$$
If we consider the problem of the minimization of this mean error, we reduce
the initial hypothesis testing problem to the problem of testing of two simple
hypotheses
\begin{align*}
&\mathscr{ H}_1\quad :\quad X^n \quad \sim\quad \Pb_{\vartheta _1},\\
&\mathscr{ H}_2\quad :\quad X^n \quad \sim\quad \PP.
\end{align*}
Then, by the Neyman-Pearson Lemma, the most powerful test in the
class~$\mathcal{ K}_\varepsilon $ $\bigl($which minimizes the mean error
$\bar\alpha \bigl(\bar\psi _n\bigr)\bigr)$ is
$$
\tilde\psi_n ^*\left(X^n\right)=\1_{\left\{\tilde L_n\left(X^n\right)>\tilde
  m_\varepsilon \right\}},
$$
where the averaged likelihood ratio
$$
\tilde L_n\left(X^n\right)= \frac{{\rm d}\PP\hfill}{{\rm d}\Pb_{\vartheta
    _1}}\left(X^n\right)=\varphi_n\int_{0}^{\varphi_n^{-1}\left(\beta
  -\vartheta _1\right)} Z_n\left(v\right)p\left(\vartheta
_1+v\varphi_n\right){\rm d}v
$$
and $\tilde m_\varepsilon $ is chosen from the condition $\tilde\psi_n^*\in
\mathcal{ K}_\varepsilon $. Now, it is clear that the BT2 $\tilde\psi_n^\star
\left(X^n\right)$ coincides with the test $\tilde\psi_n ^*\left(X^n\right)$ if
we put $\tilde m_\varepsilon =m_\varepsilon\,p\left(\vartheta _1\right)
\varphi _n $.

In the proof of the convergence in distribution of the Bayes estimator it is
shown (see \cite[Theorem 1.10.2]{IH81} and \cite{Da03}) that
$$
\varphi_n^{-1}\tilde L\left(X^n\right)\Longrightarrow p\left(\vartheta
_1\right)\int_{0}^{\infty } Z\left(v\right) {\rm d}v.
$$
Therefore (under hypothesis $\mathscr{ H}_1$),
$$
R_n\left(X^n\right)\Longrightarrow \int_{0}^{\infty }Z\left(v\right)\,{\rm d}v
$$
and the test $\tilde\psi_n^\star\left(X^n\right) $ belongs to the class
$\mathcal{ K}_\varepsilon $.

Using a similar argument, we can verify the convergence
$$
R_n\left(X^n\right)\Longrightarrow Z\left(-u_*\right)^{-1}\int_{-u_*}^{\infty
}Z\left(v\right)\,{\rm d}v
$$
under the alternative $\vartheta_{u_*}$, which concludes the proof.
\end{proof}

\subsection{Simulations}
\label{SSC}

Let us consider the following example.  We observe $n$ independent
realizations $X^n=\left(X_1,\ldots,X_n\right)$, where $ X_j=\left(X_j(t),\ t
\in \left[0,2\right]\right)$, $ j=1,\ldots,n$, of an inhomogeneous Poisson
process. The intensity function of this processes is
$$
\lambda(\vartheta,t)=2-\left|t-\vartheta\right|^{0.4} ,\qquad 0\leq t\leq 2,
$$
where the parameter $\vartheta \in \left[0.5, 2\right)$.  We take
  $\vartheta_1=1.5$ as the value of the basic hypothesis $\mathscr{ H}_1$. Of
  course it is sufficient to have simulations for the values $\vartheta \in
  \left[1.5, 2\right)$, but we consider a wider interval to show the behavior
    of the likelihood ratio on the both sides of the true value.  The Hurst
    parameter is $H=0.9$ and the constant
    $\Gamma^2_{\vartheta_1}=B(1.4,1.4)\Bigl[\frac{1}{\cos(0.4\pi)}-1\Bigr]
    \approx 1.027$.

A realization of the normalized likelihood ratio $Z_n\left(u\right)$, $u\in
\left[-5,5\right]$, and its zoom $Z_n\left(u\right)$, $u\in
\left[0.1,0.5\right]$, under the hypothesis $\mathscr{H}_1$ are given in
Figure~\ref{Z_n_u_cusp}.

\vbox{
\bigskip
\hrule
\smallskip
Here Figure~\ref{Z_n_u_cusp}
\hrule
\bigskip
}

To find the thresholds of the GLRT and of the WT, we need to find the point of
maximum and the maximal value of this function. In the case of the chosen
intensity function, the maximum is attained at one of the cusps of the
likelihood ratio (that is, on one of the events of one of the observed Poisson
processes).

It is interesting to note that if the intensity function has the same
singularity but with a different sign:
$\lambda\left(\vartheta,t\right)=0.5+\left|t-\vartheta\right|^{0.4} $, then it
is much more difficult to find the maximum (see
Figure~\ref{Z_n_u_cusp_invers}).

\vbox{
\bigskip
\hrule
\smallskip
Here Figure~\ref{Z_n_u_cusp_invers}.
\hrule
\bigskip
}

The thresholds of the GLRT, of the WT and of the BT1 are presented in
Table~\ref{Thr_cusp}.

\begin{table}[htb]
\begin{center}
\begin{tabular}{|c|c|c|c|c|c|c|c|}
  \hline
  $\varepsilon$ & 0.01 & 0.05 & 0.10 & 0.2 & 0.4 & 0.5 \\
  \hline
  $\ln h_\varepsilon$ & 2.959 & 1.641 & 1.081 & 0.559 & 0.159& 0.068 \\
  \hline
  $g_\varepsilon$ & 3.041 & 1.996 & 1.521 & 0.950 & 0.333& 0.166 \\
  \hline
  $k_\varepsilon$ & 2.864 & 2.0776 & 1.720 & 1.365 & 1.005& 0.885 \\
  \hline
\end{tabular}
\end{center}
\caption{\label{Thr_cusp}Thresholds of GLRT, WT and BT1}
\end{table}

For example, the thresholds of the GLRT are obtained by simulating $M=10^5$
trajectories of $Z^i(u)$, $u\in\left[0,20\right]$, $i=1,\ldots,M$ (when $u>20$
the value of $Z^i(u)$ is negligible), calculating for each of them the
quantity $\sup_u Z^i(u)$, and taking the ($1-\varepsilon$)M-th greatest
between them.

For the computation of the power function we calculate the frequency of
accepting the alternative hypothesis. For example, for the GLRT we use
$$
\beta\left(\hat\psi_n,u\right)\approx \frac{1}{N}\sum_{i=1}^N
\1_{\left\{\sup\limits_{v>0} Z_{n,i}(v)>h_\varepsilon\right\}}.
$$
We can see in Figure~\ref{PF_cusp_1} that, like in the regular case, for the
small values of~$u$ the power function of the WT converge more slowly than
that of the GLRT, but still more quickly than that of the BT1. When $u$ is
large, the power function of the BT1 converge more quickly than that of the
WT, and the power function of the GLRT converge the most slowly.

\vbox{
\bigskip
\hrule
\smallskip
Here Figure~\ref{PF_cusp_1}.
\hrule
\bigskip
}

Since analytic expressions for the power functions of these three tests are
not yet available, we compare them with the help of numerical simulations. It
is equally interesting to compare them to the Neyman-Pearson Test (N-PT)
constructed in the following problem of testing of two simple hypotheses.

Let us fix an alternative
$\vartheta_{u_*}=\vartheta_1+u_*\varphi_n>\vartheta_1$ and consider the
hypothesis testing problem
\begin{align*}
\mathscr{ H}_1\quad &:\qquad u  =0,\\
\mathscr{ H}_2\quad &:\qquad u =u_*.
\end{align*}
The Neyman-Pearson test is
$$
\psi_n^*\left(X^n\right)=\1_{\left\{Z_n(u_*)> d_\varepsilon\right\}},
$$
where the threshold $d_\varepsilon $ is the solution of the equation
$$
\Pb\left(Z\left(u_*\right)> d_\varepsilon\right)=\varepsilon.
$$
Recall that $Z_n\left(u_*\right)\Longrightarrow Z(u_*)$ and
$$
Z\left(u_*\right)=\exp\left\{W^H\left(u_*\right)- \frac{u_*^{2H}}2 \right\}.
$$
Hence
$$
\Pb\left(Z\left(u_*\right)>
d_\varepsilon\right)=\Pb\left\{W^H\left(u_*\right)- \frac{u_*^{2H}}2 >\ln
d_\varepsilon \right\}=\Pb\left(\zeta >\frac{\ln d_\varepsilon
  +\frac{u_*^{2H}}2}{u_*^{H}}\right)
$$
and
$$
d_\varepsilon =e^{z_\varepsilon u_*^H-\frac{u_*^{2H}}2},
$$
where $\zeta \sim \mathcal{ N}\left(0,1\right)$ and $\Pb\left(\zeta
>z_\varepsilon \right)=\varepsilon$.

Of course, it is impossible to use this N-PT in our initial problem,
since~$u_*$ (the value of $u$ under alternative) is unknown. However, as this
test is the most powerful in the class $ \mathcal{ K}_\varepsilon $, its power
(as function of $u_*$) shows an upper bound for power functions of all the
tests, and the distances between it and the power functions of the studied
tests provide useful information.

To study the likelihood ratio function under the alternative we write
$$
Z_n(u_*)=\frac{L\left(\vartheta_1+u_*\varphi_n,X^n\right)}{L\left(\vartheta_1,X^n\right)}
=\left(\frac{L\left(\vartheta_1+u_*\varphi_n-u_*\varphi_n,X^n\right)}
{L\left(\vartheta_1+u_*\varphi_n,X^n\right)}\right)^{-1}.
$$
So, for the power of the N-PT, we obtain
\begin{align*}
\beta(\psi_n^*,u_*)&=\Pb_{\vartheta_1+u_*\varphi_n}
\left(Z_n\left(u_*\right)>d_\varepsilon\right)\\
&\longrightarrow \Pb\left(Z\left(-u_*\right)^{-1}>d_\varepsilon\right)
=\Pb\left(\exp\left\{-W^H\left(-u_*\right)+
\frac{u_*^{2H}}2\right\}>d_\varepsilon\right)\\
&=\Pb\left(\zeta >\frac{\ln d_\varepsilon
    -\frac{u_*^{2H}}2}{u_*^{H}}\right)=\Pb\left(\zeta >z_\varepsilon
  -u_*^{H}\right).
\end{align*}

\vbox{
\bigskip
\hrule
\smallskip
Here Figure~\ref{PF_cusp_comp}.
\hrule
\bigskip
}

We can see that the limit power function of the GLRT is the closest one to the
limit power function of the N-PT. When $u$ is small, the limit power function
of the BT1 is lower than that of the GLRT. It becomes closer to that of the
N-PT when $u$ increases. At the same time, the limit power function of the WT
become the lowest one. Let as also mention that the limit power function of
the BT1 arrives faster to $1$ than the others (see Figure~\ref{PF_cusp_comp}).

\section{Discontinuous intensity}

Here we consider a similar hypothesis testing problem in the case of
inhomogeneous Poisson processes with discontinuous intensity function. Suppose
that the intensity function $\lambda \left(\vartheta ,t\right)$, $0\leq t\leq
\tau$, of the observed Poisson processes satisfies the following condition.

\bigskip
\textbf{\textit{S}}. \textit{The intensity function
  $\lambda(\vartheta,t)=\lambda(t-\vartheta)$, where the unknown parameter
  $\vartheta \in \Theta=\left[\vartheta _1,b\right) \subset
  \left(0,\tau\right)$, the function $\lambda(s)$, $s\in\left[-b,\tau
    -\vartheta _1\right]$, is continuously differentiable everywhere except at
  the point $t_*\in \left(-\vartheta_1,\tau-b \right)$ and this function has a
  jump at the point $t_*$ (and so, the intensity function
  $\lambda(\vartheta,t)$ has a jump at the point $t=t_*+\vartheta \in
  \left(0,\tau \right)$).}
\bigskip

We have to test the hypotheses
\begin{align*}
\mathscr{ H}_1\quad &:\qquad \vartheta =\vartheta _1,\\ \mathscr{ H}_2\quad
&:\qquad \vartheta >\vartheta _1.
\end{align*}

We study the same tests as before (GLRT, WT, BT1 and BT2), and our goal is to
chose the thresholds so, that these tests belong to the class $\mathcal{
  K}_\varepsilon $. Let us denote $\lambda(t_*+)=\lambda_+$,
$\lambda(t_*-)=\lambda_- $ and $\rho =\frac{\lambda_-}{\lambda_+}$. To compare
the power functions of the tests, we consider local alternatives which in this
problem are given by $\vartheta =\vartheta _1+u\varphi_n$,
$\varphi_n=\frac1{n\lambda_+}$. The initial problem is thus reduced to the
following one
\begin{align*}
\mathscr{ H}_1\quad &:\qquad u  =0,\\
\mathscr{ H}_2\quad &:\qquad u >0.
\end{align*}

Recall that the normalized likelihood ratio
$$
Z_n\left(u\right)=\frac{L\left(\vartheta _1+u\varphi
  _n,X^n\right)}{L\left(\vartheta _1,X^n\right)} , \qquad
u\in\mathbb{U}_n^+=\bigl[0, n \lambda_+ \left(b-\vartheta _1\right)\bigr),
$$
under the hypothesis $\mathscr{ H}_1$ converges to the process
$$
Z\left(u\right)=\exp\left\{\ln\rho\; x_*\left( u\right)-\left(\rho
-1\right)u\right\},\qquad u\geq 0,
$$
where $(x_*\left(u\right),\ u\geq 0)$ is a Poisson process of unit intensity
(see~\cite{Kut84}).

As we will see below, the limit likelihood ratio under the alternative
$\vartheta_1+u_*\varphi _n$, $u_*>0$, is
$$
Z\left(u,u_*\right)=\exp\left\{\ln\rho\; x_*\left(u,u_*\right)-\left(\rho
-1\right)u\right\},\quad u\geq 0,
$$
where $(x_*\left(u,u_*\right),\ u\geq 0)$ is a Poisson process with
``switching'' intensity function
$$
\mu \left(u,u_*\right)=\rho \,\1_{\left\{u<u_* \right\}}+\1_{\left\{u\geq
  u_*\right\}},\quad u\geq 0.
$$

Note that the limit likelihood ratio of our problem is the likelihood ratio of
a similar hypothesis testing problem ($u=0$ against $u>0$) in the case of
observations of a Poisson process $(Y\left(t\right),\ t\geq 0)$ with
``switching'' intensity function $\mu \left(t,u\right)= \rho \,\1_{\left\{t<u
  \right\}}+\1_{\left\{t\geq u\right\}}$, $t\geq 0$.

\subsection{Weak convergence}
\label{SecWC}

The considered tests (GLRT, WT, BT1 and BT2) are functionals of the likelihood
function $L\left(\cdot ,X^n\right)$. As it was shown above, all these tests
can be written as functionals of the normalized likelihood ratio
$Z_n\left(\cdot \right)$. Therefore, as in regular and cusp-type cases, we
have to prove the weak convergence of the measures induced by the normalized
likelihood ratio under hypothesis (to find the thresholds) and under
alternative (to describe the power functions).

Let $\mathscr{D}_0$ be the space of functions $z(\cdot)$ on
$\mathbb{R}_+=\left[0,+\infty\right)$ which do not have discontinuities of the
  second kind and which are such that $\lim\limits_{v\rightarrow \infty
  }z(v)=0$. We suppose that the functions $z(\cdot)\in\mathscr{D}_0$ are
  c\`adl\`ag, that is, the left limit $z(t-)=\lim\limits_{s \nearrow t}z(s)$
  exists and the right limit $z(t+)=\lim\limits_{s \searrow t}z(s)$ exists and
  equals to $z(t)$. Introduce the distance between two function
  $z_1\left(\cdot \right)$ and $z_2\left(\cdot \right)$ as
$$
d(z_1,z_2)=\inf_\nu\Bigl[\,\sup_{u\in\mathbb{R}_+}\left|z_1(u)-z_2\bigl(\nu(u)
  \bigr)\right|+\sup_{u\in\mathbb{R}_+}\left|u-\nu(u)\right|\,\Bigr],
$$
where the $\inf$ is taken over all monotone continuous one-to-one mappings
$\nu:\mathbb{R}_+\longrightarrow\mathbb{R}_+$. Let us also denote
$$
\Delta_h (z)= \sup_{u\in\mathbb{R}_+} \sup_{\delta} \biggl\{ \min
\Bigl[\bigl|z(u^{'}) - z(u)\bigr|,\bigl|z(u^{''})-z(u)\bigr|\Bigr] \biggr\} +
\sup_{|u|>1/h}\bigl|z(u)\bigr|,
$$
where the second $\sup$ is taken over the intervals $\delta=\left[
  u^{'},u^{''}\right)\, \subseteq \left[ u-h,u+h\right)$ such that
    $u\in\delta$.

 Suppose that we have a sequence $\left(Y_n\right)_{n\geq 1}$ of stochastic
 processes $\bigl($with $Y_n=(Y_n(u),\ u \in \left[0,+\infty\right))\bigr)$
   and a process $Y_0=(Y_0(u),\ u \in \left[0,+\infty\right))$ such that the
     realizations of these processes belong to the space
     $\mathscr{D}_0$. Denote ${\bf Q}_n^{(\vartheta)}$ and ${\bf
       Q}^{(\vartheta)}$ the distributions (which we suppose depending on a
     parameter $\vartheta \in \Theta $) induced on the measurable space
     $(\mathscr{D}_0,\mathcal{ B})$ by these processes. Here $\mathcal{ B} $
     is the Borel $\sigma $-algebra of the metric space $\mathscr{D}_0$.

\begin{theorem}
If, as $n \rightarrow \infty $, the finite dimensional distributions of the
process $Y_n $ converge to the finite dimensional distributions of the process
$Y_0$ uniformly in $\vartheta \in \Theta$ and for any $\delta>0$ we have
\begin{equation}
\label{11}
\lim_{h\rightarrow 0}\,\Limsup_{n \rightarrow \infty }\,\sup_{\vartheta \in
  \Theta}\,{\bf Q}_n^{(\vartheta)}\{\Delta_h (Y_n)>\delta\}=0,
\end{equation}
then ${\bf Q}_n^{(\vartheta)}\Longrightarrow{\bf Q}^{(\vartheta)}$ uniformly
in $\vartheta \in \Theta $ as $n\rightarrow \infty $.
\end{theorem}

For the proof see \cite{GS69}, Theorem 9.5.2.

Recall that such a weak convergence of the likelihood ratio process
$Z_n\left(\cdot \right)$ for the discussed model of inhomogeneous Poisson
process was already established in \cite[Sections 4.4 and 5.4.3]{Kut84} (see
as well \cite[Chapter 5]{Kut98} for similar results). The proof given there
corresponds to the weak convergence in the space
$\left(\mathscr{D}_0,\mathcal{ B}\right)$ of $Z_n\left(\cdot \right)$ under
hypothesis $\mathscr{ H}_1$.  The limit process under the alternative
$\vartheta_1+u_*\varphi_n$ is different and we study it below in order to
describe the power functions.

Let now ${\bf Q}$ be the measure induced on the measurable space
$\left(\mathscr{ D}_0, \mathcal{ B} \right)$ by the stochastic processes
$(Z\left(u,u_*\right),\ u\geq 0)$, and ${\bf Q}_n^{(\vartheta)}$ be the
measure induced (under the true value $\vartheta$) on the same space by the
processes $(Z_n\left(u\right),\ u\geq 0)$.

\begin{proposition}
Let the condition \textbf{S} be fulfilled. Then, under the alternative
$\vartheta_{u_*}=\vartheta_1+u_*\varphi_n$, we have the convergence
\begin{align}
\label{11a}
{\bf Q}_n^{(\vartheta_{u_*})}\Longrightarrow{\bf Q}.
\end{align} 
\end{proposition}

The proof is based on several lemmas, where we verify the convergence of the
finite-dimensional distributions and the condition~\eqref{11}.  As in
\cite{Kut84}, we follow the main steps of the proof of Ibragimov and
Khasminskii \cite{IH81} of a similar convergence in the case of
i.i.d.\ observations.

\begin{lemma}
\label{L1}
Let the condition \textbf{S} be fulfilled. Then, under the alternative
$\vartheta_{u_*}$, the finite-dimensional distributions of the process
$(Z_n(u),\ u\geq 0)$ converge to those of the process $(Z(u,u_*),\ u\geq 0)$.
\end{lemma}

\begin{proof}
The characteristic function of $\ln Z_n\left(u\right)$ is (see~\cite{Kut84}):
\begin{align*}
& \Ex_{\vartheta_{u_*}} \exp\left\{i\mu\ln Z_n(u)\right\} \\
&\qquad  = \exp \biggl[ n\int_0^\tau \Bigl(\exp \Bigl(i\mu \ln
    \frac{\lambda(t-\vartheta_1-u\varphi_n)}{\lambda(t-\vartheta_1)}\Bigr)-1\Bigr)
    \lambda(t-\vartheta_1-u_*\varphi_n)\, \mathrm{d} t -\\
&\qquad  \qquad -i n\mu \int_0^\tau \bigl(\lambda(t-\vartheta_1-u\varphi_n)
    -\lambda(t-\vartheta_1) \bigr)\, \mathrm{d} t \biggr]\\
&\qquad =\exp\biggl(n\int_0^\tau A_n(u,u_*,t)\,\mathrm{d} t- i n\mu \int_0^\tau
  B_n(u,u_*,t)\,\mathrm{d} t\biggr)
\end{align*}
where we denoted
\begin{align*}
A_n(u,u_*,t)&=\left[\exp \left( i\mu\ln\frac{\lambda _u}{\lambda _0}\right)-1-
  i\mu \ln\frac{\lambda _u}{\lambda _0} \right] \lambda_{u_*},\\
B_n(u,u_*,t)&=\lambda _u-\lambda _0-\lambda _0\ln\frac{\lambda _u}{\lambda
  _0}+\left(\lambda_0-\lambda_{u_*}\right) \ln\frac{\lambda _u}{\lambda _0}
\end{align*}
and $ \lambda _v=\lambda \left(t-\vartheta _1-v\varphi _n\right)$ for $v\geq
0$.

We consider two cases: $u\leq u_*$ and $u>u_*$.  Let $u\leq u_*$ and $0\leq
t\leq t_*+\vartheta _1$. Then the functions $\lambda_0,\lambda_u$ and
$\lambda_{u_*}$ are continuously differentiable and, using Taylor expansion,
we obtain the estimates
$$
\int_{0}^{t_*+\vartheta _1 }\left|A_n\left(u,u_*,t\right)\right|{\rm d}t\leq
\frac{C u^2}{n^2}\ \:\ \text{and}\ \:\ \int_{0}^{t_*+\vartheta _1
}\left|B_n\left(u,u_*,t\right)\right|{\rm d}t\leq \frac{C\left(u^2+u
  u_*\right)}{n^2}.
$$
Similar estimates hold on the interval $t\in\left[t_*+\vartheta _1+u_*\varphi
  _n,\tau \right]$. We have as well the estimates
\begin{align*}
&\int_{t_*+\vartheta _1+u\varphi _n}^{t_*+\vartheta _1+u_*\varphi _n
  }\left|A_n\left(u,u_*,t\right)\right|{\rm d}t\leq
  \frac{C u^2\left(u_*-u\right)}{n^3}\\
\intertext{and}
&\int_{t_*+\vartheta _1+u\varphi _n}^{t_*+\vartheta _1+u_*\varphi _n
}\left|B_n\left(u,u_*,t\right)\right|{\rm d}t\leq \frac{C
  u^2\left(u_*-u\right)}{n^3}+\frac{C u\left(u_*-u\right)}{n^2}.
\end{align*}
So, the main contribution comes from the integrals
\begin{align*}
n\int_{t_*+\vartheta _1}^{t_*+\vartheta _1+u\varphi _n }&\left[ \exp \left(
  i\mu\ln\frac{\lambda _u}{\lambda _0}\right)-1 \right] \lambda_{u_*}{\rm d} t
- i\mu n \int_{t_*+\vartheta _1}^{t_*+\vartheta _1+u\varphi _n }\left[\lambda
  _u-\lambda _0\right]{\rm d}t\\
&= u\left[\exp\left( i\mu \ln \frac{\lambda _-}{\lambda +}
    \right)-1\right]\frac{\lambda _-}{\lambda _+} - i u\mu
  \left[\frac{\lambda_ -}{\lambda _+}-1\right]+o\left(1\right)\\
&\longrightarrow u\left[\exp\left( i\mu \ln\rho \right)-1 \right]\rho
- i u\mu \left[\rho -1\right],
\end{align*}
and we obtain, for $u\leq u_*$,
\begin{align*}
\Ex_{\vartheta_{u_*}} \exp\left\{i\mu\ln Z_n(u)\right\} & \longrightarrow
\exp\biggl\{u\Bigl(\bigl(\exp \left\{ i\mu \ln \rho\right\}-1 \bigr) \rho
-i\mu\left(\rho-1\right) \Bigr)\biggr\}\\
& =\Ex\exp\left\{i\mu\ln Z(u,u_*)\right\}.
\end{align*}

Now we consider the case $u\geq u_*$.  As before, we obtain the convergence,
$$
n\left(\int_0^{t_*+\vartheta_1}+\int_{t_*+\vartheta_1+u\varphi_n}^\tau\right)
\bigl(\left|A_n(u,u_*,t)\right|+\left|B_n(u,u_*,t)\right|\bigr)dt
\longrightarrow 0.
$$
For the intervals $\left(t_*+\vartheta_1, t_*+\vartheta_1+u_*\varphi_n\right)$
and $\bigl(t_*+\vartheta_1+u_*\varphi_n,t_*+\vartheta_1+u\varphi_n\bigr)$, we
can write
\begin{align*}
n\int_{t_*+\vartheta_1}^{t_*+\vartheta_1+u_*\varphi_n}&\bigl(A_n(u,u_*,t)-
i\mu B_n\left(u,u_*,t\right)\bigr)\mathrm{d} t\\
& \longrightarrow \frac{u_*}{\lambda_+} \Bigl(\exp \Bigl( i\mu \ln
\frac{\lambda_-}{\lambda_+}\Bigr)-1\Bigr)\lambda_- -i\mu \bigl(\lambda_-
-\lambda_+ \bigr)\biggr)\\
&= u_*\Bigl(\bigl(\exp \left\{ i\mu \ln \rho\right\}-1 \bigr) \rho
-i\mu\left(\rho-1\right) \Bigr)
\end{align*}
and 
\begin{align*}
n\int_{t_*+\vartheta_1+u_*\varphi_n}^{t_*+\vartheta_1+u\varphi_n}&\bigl(A_n(u,u_*,t)-
i\mu B_n\left(u,u_*,t\right)\bigr) \mathrm{d} t\\
& \longrightarrow\frac{u-u_*}{\lambda_+} \biggl(\Bigl(\exp \Bigl( i\mu \ln
\frac{\lambda_-}{\lambda_+}\Bigr)-1\Bigr)\lambda_+ -i\mu \bigl(\lambda_-
-\lambda_+ \bigr)\biggr) .
\end{align*}
So, for $u>u_*$, we get
\begin{align*}
\Ex_{\vartheta_{u_*}} \exp\left\{i\mu\ln Z_n(u)\right\} & \longrightarrow
\exp\biggl\{u_*\Bigl(\bigl(\exp \left\{ i\mu \ln \rho\right\}-1 \bigr) \rho
-i\mu\left(\rho-1\right) \Bigr) \\
& \qquad\qquad +(u-u_*)\Bigl(\exp \left\{ i\mu \ln \rho\right\}-1 -i\mu
\left(\rho-1 \right)\Bigr) \biggr\}\\
& =\Ex \exp\left\{i\mu\ln Z(u,u_*)\right\}.
\end{align*}
Therefore the one-dimensional distributions of the stochastic process
$Z_n\left(\cdot \right)$ converge to those of $Z\left(\cdot ,u_*\right)$.

The convergence of arbitrary finite-dimensional distributions of
$Z_n\left(\cdot \right)$ to those of $Z\left(\cdot ,u_*\right)$ can be proved
in a similar manner.  For example, in the case of two-dimensional
distributions we can write (for $u_1<u_2<u_*$)
\begin{align*}
\Ex_{\vartheta_{u_*}}&\exp\bigl\{ i\mu _1\ln Z_n(u_1)+ i\mu _2\ln
Z_n(u_2)\bigr\}\\
&\longrightarrow\exp\biggl\{(u_2-u_1)\Bigl[\bigl(\exp\left\{ i\mu _2 \ln
  \rho\right\}-1\bigr)\rho - i\mu _2(\rho-1)\Bigr]\\
&\qquad\qquad +u_1\Bigl[\Bigl(\exp\bigl\{ i\,(\mu _1+\mu _2) \ln
  \rho\bigr\}-1\Bigr)\rho - i\,(\mu _1+\mu _2)\, (\rho-1)\Bigr]\biggr\}\\
&=\Ex\exp\bigl\{ i\mu _1\ln Z(u_1,u_*)+ i\mu _2\ln Z(u_2,u_*)\bigr\}.
\end{align*}
So, the lemma is proved.
\end{proof}

Further, we can write
$$
Z_n(u)=Z_n(u_*)\;\widetilde{Z}_n(u),
$$
where
$$
\widetilde{Z}_n(u)=\frac{\mathrm{d} \Pb_{\vartheta_1+u\varphi_n}\hfill}{\mathrm{d}
  \Pb_{\vartheta_1+u_*\varphi_n}\hfill}\qquad\text{and}\qquad
Z_n(u_*)=\frac{\mathrm{d} \Pb_{\vartheta_1+u_*\varphi_n}\hfill}{\mathrm{d}
  \Pb_{\vartheta_1}\hfill}.
$$
Note that $Z_n\left(u_*\right)$ does not depend of $u$ and we have the
convergence (under the alternative $\vartheta_{u_*}=\vartheta_1+u\varphi_n$)
\begin{equation}
\label{Zstar}
Z_n\left(u_*\right)\Longrightarrow Z_*\left(u_*\right)=\exp\left\{\ln\rho\;
x_\rho\left(u_*\right)-u_*\left(\rho -1\right)\right\},
\end{equation}
where $(x_\rho\left(u\right),\ u\geq 0)$ is a Poisson process of intensity
$\rho$. Therefore, to prove~\eqref{11a}, it is sufficient to study the
convergence of the measures induced by the stochastic process
$(\widetilde{Z}_n(u),\ u\geq 0)$.

\begin{lemma}
Let the conditions \textbf{S} be fulfilled.  Then there exists a constant
$C>0$, such that
\begin{equation}
\label{15}
\Ex_{\vartheta_1+u_*\varphi_n}\bigl|\widetilde{Z}_n^{1/2}(u_1)-\widetilde{Z}_n^{1/2}(u_2)
\bigr|^2\leq C\left|u_1-u_2\right|
\end{equation}
for all $u_*,u_1,u_2\in\mathbb{U}_n^+$.
\end{lemma}

\begin{proof}
According to~\cite[Lemma 1.1.5]{Kut98}, we have (for $v_1>v_2>0$)
\begin{align*}
&\Ex_{\vartheta_1+u_*\varphi_n}\bigl|\widetilde{Z}_n^{1/2}(u_1)-\widetilde{Z}_n^{1/2}(u_2)\bigr|^2\\
&\ \,\leq
  \int_{0}^{n\tau}\left(\frac{\lambda^{1/2}(t-\vartheta_1-u_1\varphi_n)}
      {\lambda^{1/2}(t-\vartheta_1-u_*\varphi_n)}
      -\frac{\lambda^{1/2}(t-\vartheta_1-u_2\varphi_n)}
      {\lambda^{1/2}(t-\vartheta_1-u_*\varphi_n)}\,\right)^2\lambda
      (t-\vartheta_1-u_*\varphi_n)\,\mathrm{d} t\\
&\ \,=n\int_{0}^{\tau}\bigl(\lambda^{1/2}(t-\vartheta _1-u_1\varphi_n)
          -\lambda^{1/2}(t-\vartheta _1-u_2\varphi_n)\,\bigr)^2\,\mathrm{d} t\\
&\ \,=n\left(\int_{0}^{t_*+u_2\varphi_n}+
          \int_{t_*+u_2\varphi_n}^{t_*+u_1\varphi_n}+
          \int_{t_*+u_1\varphi_n}^{\tau}\right) \bigl(\lambda_{u_1}^{1/2}
          -\lambda_{u_2}^{1/2}\,\bigr)^2\,\mathrm{d} t\\
&\ \,=n(I_1+I_2+I_3)
\end{align*}
with obvious notations.

As the functions $\lambda_{u_1} $ and $\lambda_{u_2}$ are continuously
differentiable on the intervals $\left[0,t_*+u_2\varphi_n\right]$ and
$\left[t_*+u_1\varphi_n,\tau\right]$, we can write
$$
\lambda^{\frac12}\left(\vartheta_1+u_1\varphi_n,t\right)-
\lambda^{\frac12}\left(\vartheta_1+u_2\varphi_n,t\right)
=\frac{\left(u_1-u_2\right)\varphi_n}2\,
\frac{\dot\lambda\left(\vartheta_v,t\right)}
     {\lambda^{\frac12}\left(\vartheta_v,t\right)}\,,
$$
where $v$ is some intermediate point between $u_1$ and $u_2$. Therefore
\begin{align*}
n\left(I_1+I_3\right) &\leq
n\varphi_n^2\frac{(u_1-u_2)^2}4\left(\int_{0}^{t_*+u_2\varphi_n}+\int_{t_*+u_1\varphi_n}^{\tau}\right)
\frac{\dot\lambda\left(\vartheta_v,t\right)^2}{\lambda\left(\vartheta_v,t\right)}\,\mathrm{d}
t\\
& \leq \frac{ C}{n\lambda_+^2}\left|u_1-u_2\right|^2 \leq {C}
\left|u_1-u_2\right|
\end{align*}
where we took into account the inequality $\left|u_1-u_2\right|\leq C n$.

Since the function $\lambda$ is bounded, we have the estimate
$$
n I_2 \leq
n\frac{\left|u_1-u_2\right|}{n\lambda_+}C=\frac{C}{\lambda_+}\left|u_1-u_2\right|,
$$
and so the inequality \eqref{15} holds with some constant $C>0$.
\end{proof}

\begin{lemma}
Let the condition \textbf{S} be fulfilled. Then there exists a constant
$k^*>0$ such that
\begin{equation}
\label{16}
\Ex_{\vartheta_1+u_*\varphi_n}\widetilde{Z}_n^{1/2}(u)\leq\exp\bigl\{-k^*\left|u-u_*\right|\bigr\}
\end{equation}
for all $u^*,u\in\mathbb{U}_n^+$.
\end{lemma}

\begin{proof}
According to~\cite[Lemma 1.1.5]{Kut98}, we have
\begin{align*}
&\Ex_{\vartheta_1+u_*\varphi_n}\widetilde{Z}_n^{1/2}(u)\\
&\qquad =\exp\Bigl\{-\frac{n}{2}\int_0^{\tau}
  \Bigl(\frac{\lambda^{1/2}(t-\vartheta_1-u\varphi_n)}{\lambda^{1/2}(t-\vartheta_1-u_*\varphi_n)}
  -1\,\Bigr)^2\,\lambda(t-\vartheta_1-u_*\varphi_n)\,\mathrm{d} t\Bigr\}\\
&\qquad
  =\exp\Biggl\{-\frac{n}{2}\,\int_0^{\tau}\Bigl(\lambda^{1/2}\bigl(t-\vartheta_1-u\varphi_n\bigr)
  -\lambda^{1/2}(t-\vartheta_1-u_*\varphi_n)\,\Bigr)^2\,\mathrm{d} t\Biggr\}\\
&\qquad =\exp \left\{-\frac{n}{2}\,F_n(u,u_*)\right\}
\end{align*}
with an obvious notation.

Let us consider separately the cases $u\in
D=\left\{v:\;\left|v-u_*\right|<\delta n\lambda _+\right\}$ and $u\in
D^c=\left\{v:\; \left|v-u_*\right|\geq \delta n\lambda _+\right\}$.  Here
$\delta$ is some positive constant which will be chosen later. For simplicity
we suppose that $u>u_*$.

For the values $u\in D$ we have
\begin{align*}
n F_n(u,u_*)&\geq n
\int_{t_*+\vartheta_1+u_*\varphi_n}^{t_*+\vartheta_1+u\varphi_n}
\left[\lambda^{1/2}\bigl(t-\vartheta_1-u\varphi_n\bigr)
  -\lambda^{1/2}(t-\vartheta_1-u_*\varphi_n)\right]^2\,\mathrm{d} t\\
&\geq \frac{\left|u-u_*\right|}{\lambda _+}\inf_{t_*+u_*\varphi _n\leq s\leq
  t_*+u\varphi _n } \left[\lambda^{1/2}\bigl(s-u\varphi_n\bigr)
  -\lambda^{1/2}(s-u_*\varphi_n)\right]^2\\
&\geq \frac{\left|u-u_*\right|}{2\lambda _+}
\left(\sqrt{\lambda_-}-\sqrt{\lambda_+}\right)^2=\frac{\left|u-u_*\right|}{2}
\left(\sqrt{\rho }-1\right)^2
\end{align*}
for sufficiently small $\delta $.

Further, note that for any $\nu >0$ we have
$$
g\left(\nu \right)=\inf_{\left|s-s_0\right|>\nu } \int_{0}^{\tau
}\left[\sqrt{\lambda \left(t-\vartheta _1-s\right)}-\sqrt{\lambda
    \left(t-\vartheta _1-s_0\right)}\right]^2{\rm d}t >0.
$$
Indeed, if $g\left(\nu \right)=0 $ then for some $s_*$ we have $\lambda
\left(t-\vartheta _1-s_*\right)=\lambda \left(t-\vartheta _1-s_0\right) $ for
all $t\in \left[0,\tau \right]$, but this equality for discontinuous $\lambda
\left(\cdot \right)$ and all $t$ is impossible. Hence, for the values $u\in
D^c$ we have
$$
n F_n\left(u,u_*\right)\geq n g\left(\delta \right)\geq \frac{g\left(\delta
  \right) \left|u-u_*\right|}{C}
$$
where we took into account the inequality $\left|u-u_*\right|\leq C n$.

So, the inequality \eqref{16} is proved.
\end{proof}

The presented estimates \eqref{15}, \eqref{16} and Lemma \ref{L1} allow us to
finish the proof following the same lines as it was done in \cite{Kut84},
Section 5.4.3.

\subsection{GLRT}

The GLRT is based on the statistic
$$
Q_n\left(X^n\right)=\sup_{\vartheta \geq \vartheta
  _1}L\left(\vartheta,\vartheta_1,X^n\right) =
\max\left[L\left(\hat\vartheta_n+,\vartheta_1,X^n\right),\,
  L\left(\hat\vartheta_n-,\vartheta_1,X^n\right)\right]
$$
(where $\hat\vartheta _n$ is the MLE) and is of the form
$$
\hat\psi _n\left(X^n\right)=\1_{\left\{Q_n\left(X^n\right)> h_\varepsilon
  \right\}}.
$$
The threshold $h_\varepsilon $ is defined with the help of the convergence
(under hypothesis $\mathscr{ H}_1$)
$$
Q_n\left(X^n\right)=\sup_{u\in\mathbb{U}_n^{+}}Z_n\left(u\right)\Longrightarrow
\sup_{u> 0}Z\left(u\right)=\hat Z.
$$
Hence $h_\varepsilon=h_\varepsilon\left(\rho \right) $ is solution of the
equation
$$
\Pb\left\{\hat Z>h_\varepsilon \right\}=\varepsilon.
$$

Let us fix an alternative $\vartheta_{u_*}=\vartheta _1+u_*\varphi_n$,
$u_*>0$. Then for the power function we have
\begin{align*}
\beta \left(\hat\psi_n,u_*\right)=\Ex_{\vartheta _{u_*}}
\hat\psi_n\left(X^n\right)&=\Pb_{\vartheta _{u_*}}
\left\{\sup_{u>0}Z_n\left(u\right)>h_\varepsilon \right\}\\
&\longrightarrow \Pb\left\{\sup_{u>0}Z\left(u,u_*\right)>h_\varepsilon
\right\}.
\end{align*}
Putting $Y\left(u\right)=\ln Z\left(u,u_*\right)=\ln \rho\,
x_*\left(u,u_*\right)-\left(\rho -1\right)u$, we can write
$$
\sup_{u>0} Y(u) =\max\left( \sup_{0<u<u_*}Y\left(u\right),\:Y\left(u_*\right)+
\sup_{u\geq u_*}\left[Y\left(u\right)-Y\left(u_*\right)\right] \right).
$$
Note that the Poisson process $\tilde
x\left(v\right)=x_*\left(u_*+v,u_*\right)-x_*\left(u_*,u_*\right)$, $v\geq 0$,
is independent from $(x_*\left(u,u_*\right),\ 0\leq u\leq u_*)$. Hence we can
write the following representation of the limit power function:
$$
\beta \left(\hat\psi_n,u_*\right)\longrightarrow
\Pb\left\{\max\left(\sup_{0<u<u_*}Z_*\left(u\right),\:
Z_*\left(u_*\right)\,\tilde Z \right)>h_\varepsilon \right\},
$$
where the random variable $\tilde Z=\sup\limits_{v\geq 0} \exp\left\{{\ln \rho
  \,\tilde x_*\left(v\right)}-\left(\rho -1\right)v\right\}$ is independent
from $(Z_*\left(u\right),\ 0\leq u\leq u_*)$, and the process
$Z_*\left(\cdot\right)$ is defined as in~\eqref{Zstar}. Let us note that this
expression is useful for numerical simulation of the power function. It
simplifies the calculations since the simulated values of $\tilde Z$ can be
used many times for different values of $u_*$.

\subsection{Wald test}

The Wald test is based on the MLE $\hat\vartheta _n$. We already know that
$$
\varphi_n^{-1}\left(\hat\vartheta _n-\vartheta _1\right)\Longrightarrow \hat
u,
$$
where $\hat u$ is defined by the equation
$$
\max\left[Z\left(\hat u+\right),Z\left(\hat u-\right)\right]=\sup_{u>0}
Z\left(u\right).
$$
The Wald test is
$$
\psi_n^\circ\left(X^n\right)=\1_{\left\{\varphi_n^{-1}\left(\hat\vartheta
  _n-\vartheta _1\right)>g_\varepsilon \right\}},
$$
where the threshold $g_\varepsilon =g_\varepsilon \left(\rho \right)$ is
solution of the equation
$$
\Pb\left\{\hat u>g_\varepsilon \right\}=\varepsilon.
$$

For the power function we have (below $\vartheta_{u_*}=\vartheta
_1+u_*\varphi_n$)
\begin{align*}
\beta \left(\psi _n^\circ,u_*\right)&=\Ex_{\vartheta_{u_*}}\psi_n^\circ
\left(X^n\right)=\Pb_{\vartheta_{u_*}} \left\{\varphi_n^{-1}\left(\hat\vartheta
_n-\vartheta _1\right)>h_\varepsilon \right\}\\
&=\Pb_{\vartheta_{u_*}} \left\{\sup_{\varphi_n^{-1}\left(\theta -\vartheta
  _1\right)>h_\varepsilon } L\left(\theta ,X^n\right) >
\sup_{\varphi_n^{-1}\left(\theta -\vartheta _1\right)\leq h_\varepsilon }
L\left(\theta ,X^n\right) \right\}\\
&=\Pb_{\vartheta_{u_*}} \left\{\sup_{\varphi_n^{-1}\left(\theta -\vartheta
  _1\right)>h_\varepsilon }\frac{ L\left(\theta ,X^n\right)}{ L\left(\vartheta
  _1 ,X^n\right) } > \sup_{\varphi_n^{-1}\left(\theta -\vartheta _1\right)\leq
  h_\varepsilon } \frac{L\left(\theta ,X^n\right) }{ L\left(\vartheta _1
  ,X^n\right) } \right\}\\
&=\Pb_{\vartheta_{u_*}} \left\{\sup_{u>h_\varepsilon }Z_n\left(u\right) >
\sup_{u\leq h_\varepsilon }Z_n\left(u\right) \right\}\\
 &\longrightarrow \Pb\left\{\sup_{u>h_\varepsilon }Z\left(u,u_*\right) >
\sup_{u\leq h_\varepsilon }Z\left(u,u_*\right) \right\}=\Pb \left\{\hat
u^*>h_\varepsilon \right\},
\end{align*}
where the random variable $\hat u^*$ is defined by the equation
$$
\max \left[Z\left(\hat u^*+,u\right),Z\left(\hat
  u^*-,u\right)\right]=\sup_{u\geq 0 }Z\left(u,u_*\right).
$$

Let us note, that we can also give another representation of the power
function using the limit (under the alternative $\vartheta_{u_*}$) of the
normalized likelihood ratio
$L(\vartheta_{u_*}+u\varphi_n,\vartheta_{u_*},X^n)$, $u\geq -u_*$. This limit
is the stochastic process $(Z^\star\left(u\right),\ u\geq -u_*)$ defined by
$$
Z^\star\left(u\right)=\begin{cases}
\exp\bigl\{\ln\rho\; x_*\left(u\right)-\left(\rho -1\right)u\bigr\}& \text{if
}u\geq 0,\vphantom{\Big(}\\
\exp\bigl\{-\ln\rho\; x_\rho\bigl((-u)-\bigr)-\left(\rho -1\right)u\bigr\}&
\text{if }-u_*\leq u\leq 0\vphantom{\Big(},
\end{cases}
$$
where $(x_*\left(u\right),\ u\geq 0)$ and $(x_\rho\left(u\right),\ u\geq 0)$
are Poisson processes of unit intensity and of intensity $\rho$
respectively. Note that in~\cite{Kut84} this limit was established for a fixed
value $\vartheta$ but, taking into account Section~\ref{SecWC}, it clearly
holds for ``moving'' value $\vartheta_{u_*}$.  Note also that the positive
real axis part $(Z^\star\left(u\right),\ u\geq 0)$ of the process
$Z^\star(\cdot)$ is nothing but the process $Z\left(\cdot\right)$.

Now, we have
\begin{align*}
\beta \left(\psi_n^\circ,u_*\right)=\Ex_{\vartheta _{u_*}}\psi
_n^\circ\left(X^n\right)&=\Pb_{\vartheta _{u_*}}
\left\{\varphi_n^{-1}\left(\hat\vartheta _n-\vartheta
_{u_*}\right)+u_*>h_\varepsilon \right\}\\
&\longrightarrow \Pb\left\{\hat u_*>h_\varepsilon-u_*\right\}
\end{align*}
where $\hat u_*$ is defined by the equation
$$
\max\left[Z^\star\left(\hat u_*+\right),Z^\star\left(\hat
  u_*-\right)\right]=\sup_{u\geq -u_*} Z^\star\left(u\right).
$$

\subsection{Bayes tests}

Suppose that the parameter $\vartheta $ is a random variable with known
probability density $p\left(\theta \right)$, $\vartheta _1\leq \theta
<b$. This function is supposed to be continuous and positive.

We consider two Bayes tests. The first one is based on the Bayes estimator,
while the second one is based on the averaged likelihood ratio.

\bigskip
\bigskip

The first test, which we call BT1, is similar to WT, but is based on the Bayes
estimator (BE) $\tilde\vartheta _n $ rather than on the MLE:
$$
\tilde\psi_n\left(X^n\right)=\1_{\left\{\varphi_n^{-1}\left(\tilde\vartheta
  _n-\vartheta _1\right) >k_\varepsilon \right\}},\qquad\tilde\vartheta
_n=\frac{\int_{\vartheta_1}^{b }\theta
  p\left(\theta\right)L\left(\theta,\vartheta_1,X^n\right){\rm
    d}\theta}{\int_{\vartheta_1}^{b }\theta
  L\left(\theta,\vartheta_1,X^n\right){\rm d}\theta}.
$$
The properties of the likelihood ratio established in Lemmas 1--3 allow us to
justify the limit
$$
\Ex_{\vartheta _1}\tilde\psi_n\left(X^n\right)\longrightarrow \Pb\left\{\tilde
u>k_\varepsilon \right\},\qquad \tilde u=\frac{\int_{0}^{\infty
  }v Z\left(v\right){\rm d}v}{\int_{0}^{\infty }Z\left(v\right){\rm d}v}.
$$
The proof follows from the general results concerning the Bayes estimators
described in~\cite{IH81} (see as well~\cite{Kut84}).

For the power function, using the convergence (under the alternative
$\vartheta_{u_*}$) of the process $Z_n\left(\cdot\right)$ to the process
$Z\left(\cdot,u_*\right)$, we obtain
$$
\beta \left(\tilde\psi _n,u_*\right)=\Pb_{\vartheta
  _{u_*}}\left\{\varphi_n^{-1}\left(\tilde\vartheta _n-\vartheta _1\right)
>k_\varepsilon \right\}\longrightarrow\Pb\left\{\frac{\int_{0}^{\infty }v
  Z\left(v,u_*\right){\rm d}v}{\int_{0}^{\infty }Z\left(v,u_*\right){\rm d}v}
>k_\varepsilon \right\}.
$$

Let us note, that we can also give another representation of the limit power
function using the process $Z^\star\left(\cdot\right)$. We have the
convergence (under the alternative $\vartheta_{u_*}$)
$$
\varphi_n^{-1}\left(\tilde\vartheta _n-\vartheta
_1\right)=\varphi_n^{-1}\left(\tilde\vartheta _n-\vartheta
_{u_*}\right)+u_*\Longrightarrow \frac{\int_{-u_*}^{\infty }v
  Z^\star\left(v\right){\rm d}v}{\int_{-u_*}^{\infty
  }Z^\star\left(v\right){\rm d}v}+u_*.
$$
Hence
$$
\beta \left(\tilde\psi _n,u_*\right)\longrightarrow\Pb\left\{
\frac{\int_{-u_*}^{\infty }v Z^\star\left(v\right){\rm
    d}v}{\int_{-u_*}^{\infty }Z^\star\left(v\right){\rm d}v}
>k_\varepsilon-u_*\right\}.
$$

\bigskip
\bigskip

The second test, which we call BT2, is the test which minimizes the mean error
of the second kind. We have
$$
\varphi_n^{-1}\tilde L\left(X^n\right)=\varphi_n^{-1}\int_{\vartheta
  _1}^{b}p\left(\theta \right)\frac{L\left(\theta
  ,X^n\right)}{L\left(\vartheta_1 ,X^n\right) }\,{\rm d}\theta \Longrightarrow
p\left(\vartheta _1\right)\int_{0}^{\infty } Z\left(v\right)\;{\rm d}v.
$$
Hence the test
$$
\tilde\psi_n^\star\left(X^n\right)=\1_{\left\{R_n>m_\varepsilon
  \right\}},\qquad R_n=\frac{\tilde L\left(X^n\right) }{\varphi_n\;
  p\left(\vartheta _1\right)}\,,
$$
with threshold $m_\varepsilon $ satisfying the equation
$$
\Pb\left\{ \int_{0}^{\infty } Z\left(v\right){\rm d}v >m_\varepsilon
\right\}=\varepsilon
$$
belongs to the class $\mathcal{ K}_\varepsilon$ and minimizes the mean error
of the second kind.

\subsection{Simulations}

We consider $n$ independent realizations $X_j=(X_j(t),\ t \in
\left[0,4\right])$, $j=1,\ldots,n$, of a Poisson process of intensity function
$$
\lambda(t,\vartheta)=\lambda(t-\vartheta)=3-2\cos^2(t-\vartheta)\1_{\left\{t\geq
  \vartheta\right\}},\qquad 0\leq t\leq 4.
$$
We take $\vartheta _1=3$ and $b=4$, and therefore $\lambda _+=1$, $\lambda
_-=3$ and $\rho =\frac{\lambda _-}{\lambda _+}=3$. The log-likelihood ratio is
\begin{align*}
\ln Z_n(u)&=\sum_{j=1}^n\int_{3}^{3+u/n} \ln \frac3{3-2\cos^2(t-3)}\,{\rm
  d}X_j(t)\\
&\qquad+\sum_{j=1}^n\int_{3+u/n}^4 \ln
\frac{3-2\cos^2(t-3-u/n)}{3-2\cos^2(t-3)}\,{\rm d}X_j(t)\\
&\qquad-u-\frac{n}2 \sin\bigl(2\bigr)+\frac{n}2 \sin\bigl(2(1-u/n)\bigr).
\end{align*}

Recall that in this case the limit (under $\mathscr{H}_1$) of the likelihood
ratio is
$$
Z\left(u\right)=\exp\left\{\ln 3\; x_*\left(u\right)-2u\right\},
$$
where $(x_*\left(u\right),\ u\geq 0)$ is a Poisson process of unit
intensity. A realization of this limit likelihood ratio $\bigl($or, more
precisely, of the logarithm of its two-sided version
$Z^\star\left(\cdot\right)\bigr)$ and its zoom are given in
Figure~\ref{fct_t_dis}.

\vbox{
\bigskip
\hrule
\smallskip
Here Figure~\ref{fct_t_dis}.
\hrule
\bigskip
}

Using this limit we obtain the threshold $h_\varepsilon $ of the GLRT as
solution of the equation
$$
\Pb\left\{\sup_{u> 0}Z(u)>h_\varepsilon \right\}=\varepsilon.
$$

It is convenient for simulations to transform the limit process as follows:
$$
\exp\left[\ln 3\left(x_*\left(u\right)-\frac2{\ln 3}u\right)\right]
=\exp\left\{\ln 3\left[\Pi \left(\frac2{\ln 3}u\right)-\frac2{\ln
    3}u\right]\right\}
$$
where $\Pi \left(\cdot \right)$ is a Poisson process of intensity $\gamma
=\frac{\ln 3}2<1$.  Hence, the threshold~$h_\varepsilon$ is determined by the
equation
$$
\Pb\left\{\sup_{t> 0}\left[\Pi \left(t\right)-t\right]> \frac{\ln
  h_\varepsilon}{\ln 3} \right\}=\varepsilon .
$$
The distribution of $\sup\limits_{t>0}\left[\Pi\left(t\right)-t\right]$ is
given by the well-known formula
$$
\Pb\left\{\sup_{t> 0}\left[\Pi\left(t\right)-t\right]\geq
x\right\}=\sum_{m>x}\frac{(m-x)^m}{m!}(\gamma e^{-\gamma})^m e^{\gamma
  x}(1-\gamma)
$$
obtained by Pyke in~\cite{Pyke59}.

Note that there is equally an analytic expression for the distribution of the
random variable $\hat t=\argmax_{t\geq 0}\left[\Pi\left(t\right)-t\right]
$. This expression was obtained by Pflug in~\cite{Pflug93} and is given by
$$
\Pb\left\{\hat t<z\right\}=\Pb\left\{\sum_{k=1}^{\nu }\eta _k <z\right\},
$$
where $\left\{\eta_k\right\}_{k\in\mathbb{N}^*}$ is an i.i.d.\ sequence with
common distribution
$$
\Pb\left\{\eta_k\leq x\right\}=\frac1{\gamma}\left[1-(1-\gamma)e^{-\gamma
    x}\sum_{j=0}^{[x]-1}\frac{(\gamma x)^j}{j\,!}-e^{-\gamma x}\frac{(\gamma
    x)^{[x]}}{[x]\,!}\right],
$$
$\nu$ is a random variable independent of $\eta _k$, $k\in\mathbb{N}^*$, and
distributed according to geometric law
$$
\Pb\left\{\nu=i\right\}=(1-\gamma)\gamma^i,\quad i\in\mathbb{N},
$$
and we use the convention $\sum_{k=1}^0\eta_k=0$. Now, for the threshold
$g_\varepsilon$ of the WT we can write
$$
\Pb\left\{\hat u>g_\varepsilon \right\}=\Pb\left\{\argmax_{t\geq
  0}\left[\Pi\left(t\right)-t\right]>\frac{\ln 3}2
g_\varepsilon\right\}=\varepsilon.
$$
However, the numerical solution of this equation is not easy, and it is
simpler to obtain the threshold $g_\varepsilon$ by Monte Carlo simulations.

The thresholds are presented in Table~\ref{Thr_dis}.

\begin{table}[htb]
\begin{center}
\begin{tabular}{|c|c|c|c|c|c|c|}
 \hline
 $\varepsilon$ & 0.01 & 0.05 & 0.10 & 0.20 & 0.40 & 0.50 \\
 \hline
 $\ln h_\varepsilon$ &4.242 & 2.607 & 1.922 & 1.120 & 0.573 & 0.191 \\
 \hline
 $g_\varepsilon$ & 5.990 & 3.556 & 2.078 & 1.045 & 0.329 & 0.099 \\
 \hline
 $k_\varepsilon$ & 6.669 & 3.937 & 2.983 & 2.132 & 1.402 & 1.196\\
 \hline
\end{tabular}
\caption{\label{Thr_dis}Thresholds of GLRT, WT and BT1}
\end{center}
\end{table}

\vbox{
\bigskip
\hrule
\smallskip
Here Figure~\ref{PF_dis}.
\hrule
\bigskip
}

It is interesting to compare the studied tests with the Neyman-Pearson test
(N-PT) corresponding to a fixed value $u_*$ of $u$. Of course, it is
impossible to use this N-PT in our initial problem, since $u_*$ (the value of
$u$ under alternative) is unknown. Nevertheless, its power (as function of
$u_*$) shows an upper bound for power functions of all the tests, and the
distances between it and the power functions of studied tests provide an
important information. Let us fix some value $u_*>0$ and introduce the N-PT
$$
\psi_n^*\left(X^n\right)=\1_{\left\{Z_n(u_*)>
  d_\varepsilon\right\}}+q_\varepsilon\1_{\left\{Z_n(u_*)=
  d_\varepsilon\right\}},
$$
where $d_\varepsilon$ and $q_\varepsilon$ are solutions of the equation
$$
\Pb\left(Z\left(u_*\right)>
d_\varepsilon\right)+q_\varepsilon\Pb\left(Z\left(u_*\right)=
d_\varepsilon\right)=\varepsilon.
$$
Denoting $D_\varepsilon=\left(\ln d_\varepsilon+(\rho-1)u_*\right)/\ln\rho$,
we can rewrite this equation as
$$
\Pb\left(x_*(u_*)>D_\varepsilon\right)
+q_\varepsilon\Pb\left(x_*(u_*)=D_\varepsilon\right)=\varepsilon.
$$
Here $x_*\left(u_* \right)$ is a Poisson random variable with parameter $u_*$,
and so the quantities $D_\varepsilon $ and $q_\varepsilon $ can be computed
numerically.

A similar calculation yields the limit power of the N-PT:
$$
\beta\left(\psi_n^*,u_*\right)\longrightarrow
\Pb\left(x_*(u_*,u_*)>D_\varepsilon\right)
+q_\varepsilon\Pb\left(x_*(u_*,u_*)=D_\varepsilon\right).
$$
where $x_*\left(u_*,u_*\right)$ is a Poisson random variable with parameter
$\rho u_*$.

\vbox{
\bigskip
\hrule
\smallskip
Here Figure~\ref{PF_dis_rho_3}.
\hrule
\bigskip
}

The results of simulations are presented in Figure~\ref{PF_dis_rho_3} for two
cases: $\varepsilon =0.05$ and $\varepsilon=0.4$. In both cases the limit
power function of the GLRT is the closest one to the limit power of the N-PT,
and the limit power function of the BT1 arrives faster to $1$ than the others.

\section{Acknowledgements}

This study was partially supported by Russian Science Foundation (research
project No. 14-49-00079). The authors thank the Referee for helpful comments.

\begin{figure}[htb]
\begin{center}
\includegraphics[width=\textwidth]{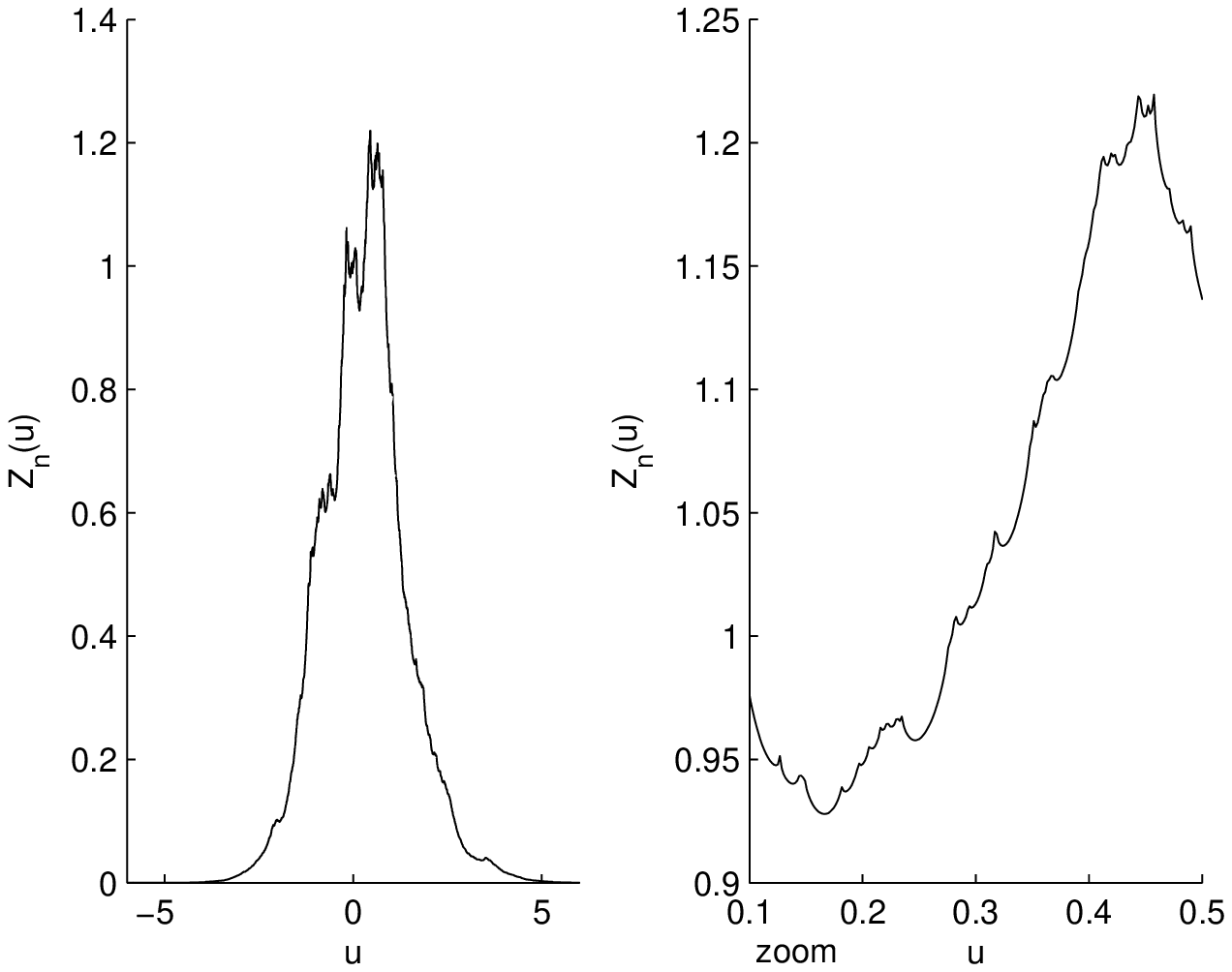}
\caption{\label{Z_n_u_cusp}A realization of $Z_n(\cdot)$ with
  $\lambda\left(\vartheta,t\right)=2-\left|t-1.5\right|^{0.4}$ and $n=1000$}
\end{center}
\end{figure}

\begin{figure}[htb]
\begin{center}
\includegraphics[width=\textwidth]{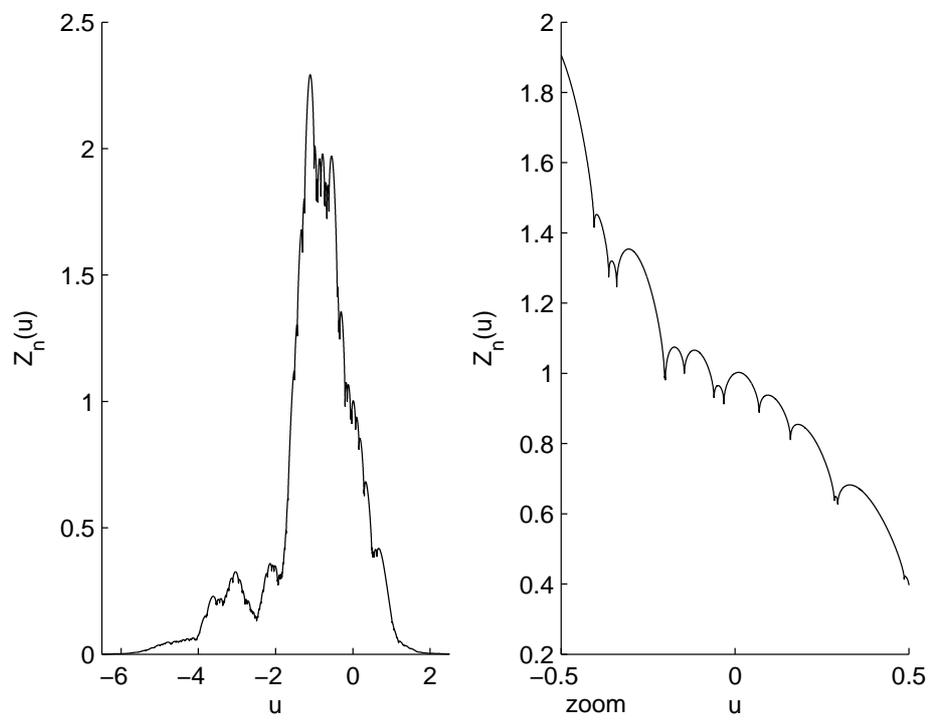}
\caption{\label{Z_n_u_cusp_invers}A realization of $Z_n(\cdot)$ with
  $\lambda\left(\vartheta,t\right)=0.5+\left|t-1.5\right|^{0.4}$ and $n=1000$}
\end{center}
\end{figure}

\begin{figure}[htb]
\begin{center}
\includegraphics[width=\textwidth]{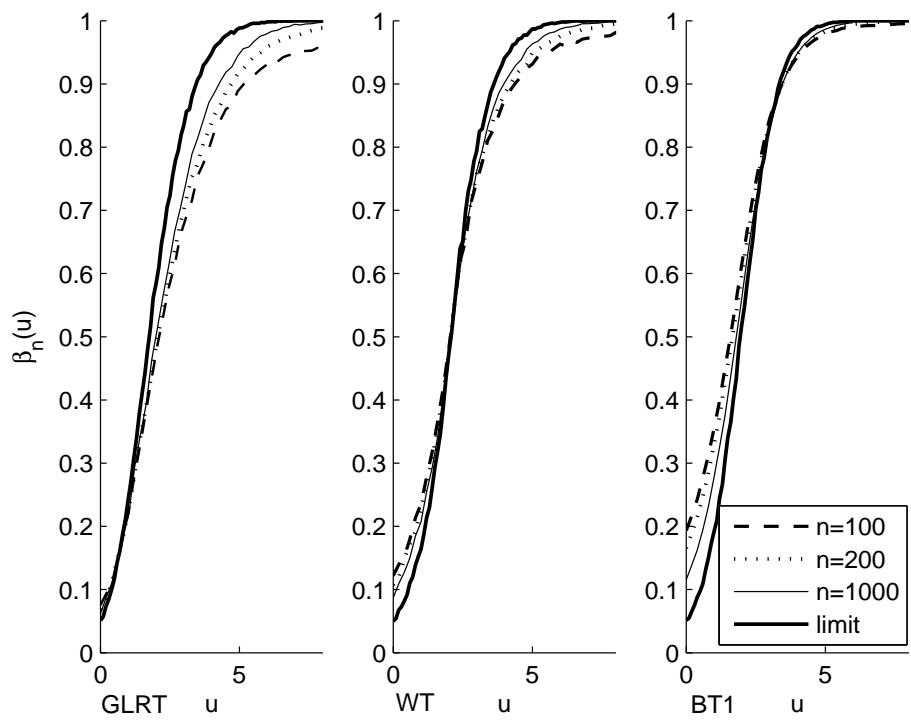}
\caption{\label{PF_cusp_1}Power functions of GLRT, WT and BT1 in cusp case
  with $\lambda\left(\vartheta,t\right)=2-\left|t-\vartheta\right|^{0.4}$}
\end{center}
\end{figure}

\begin{figure}[htb]
\begin{center}
\includegraphics[width=\textwidth]{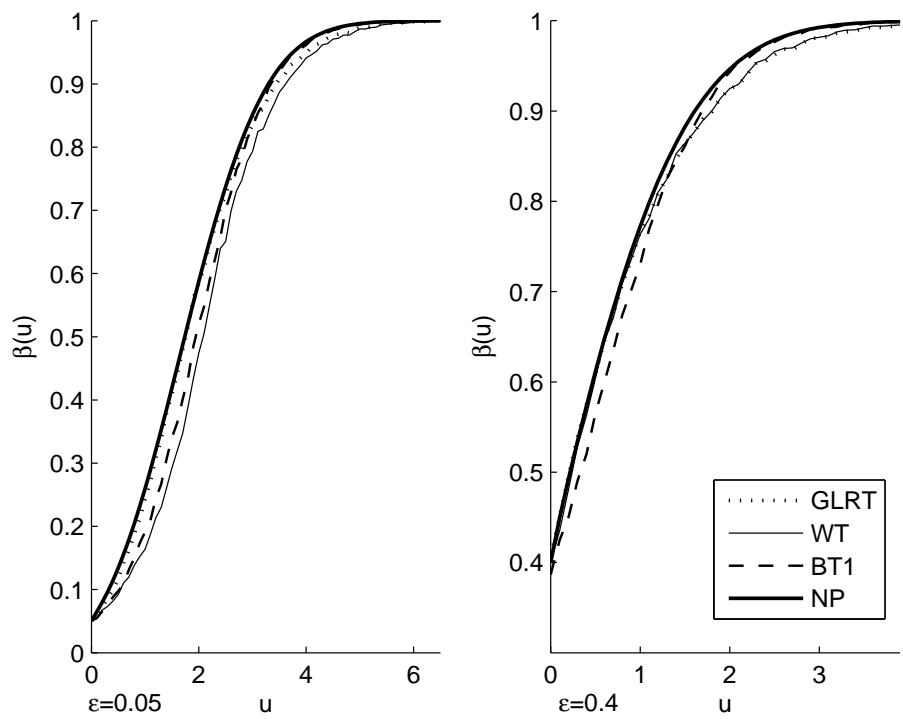}
\caption{\label{PF_cusp_comp}Comparison of limit power functions in cusp case
  with $\lambda\left(\vartheta,t\right)=2-\left|t-\vartheta\right|^{0.4}$}
\end{center}
\end{figure}

\begin{figure}[htb]
\begin{center}
\includegraphics[width=\textwidth]{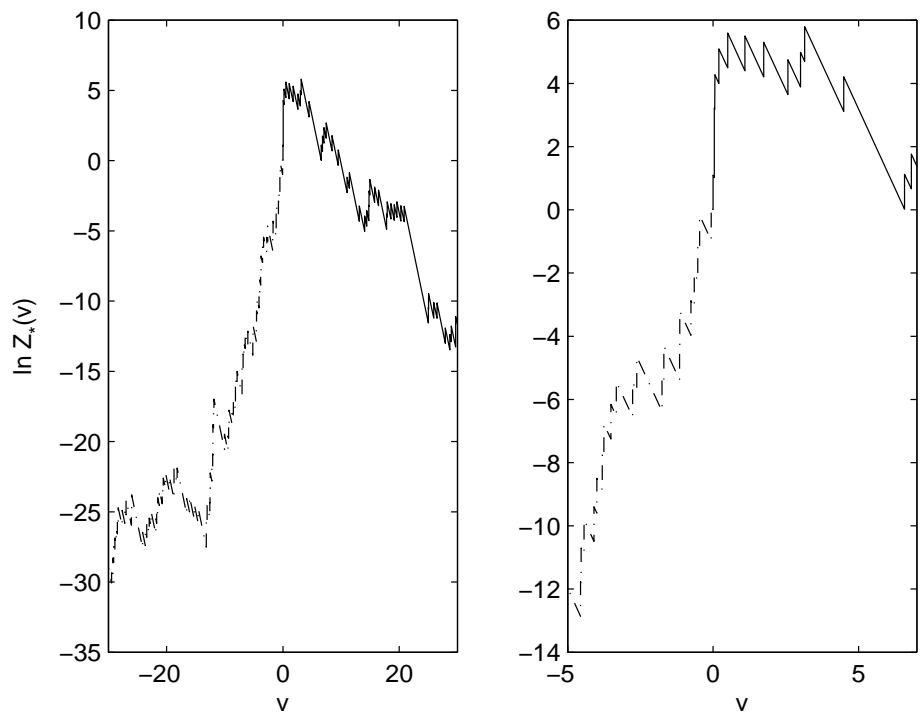}
\caption{\label{fct_t_dis}A realization of $\ln Z^\star(\cdot)$}
\end{center}
\end{figure}

\begin{figure}[htb]
\begin{center}
\includegraphics[width=\textwidth]{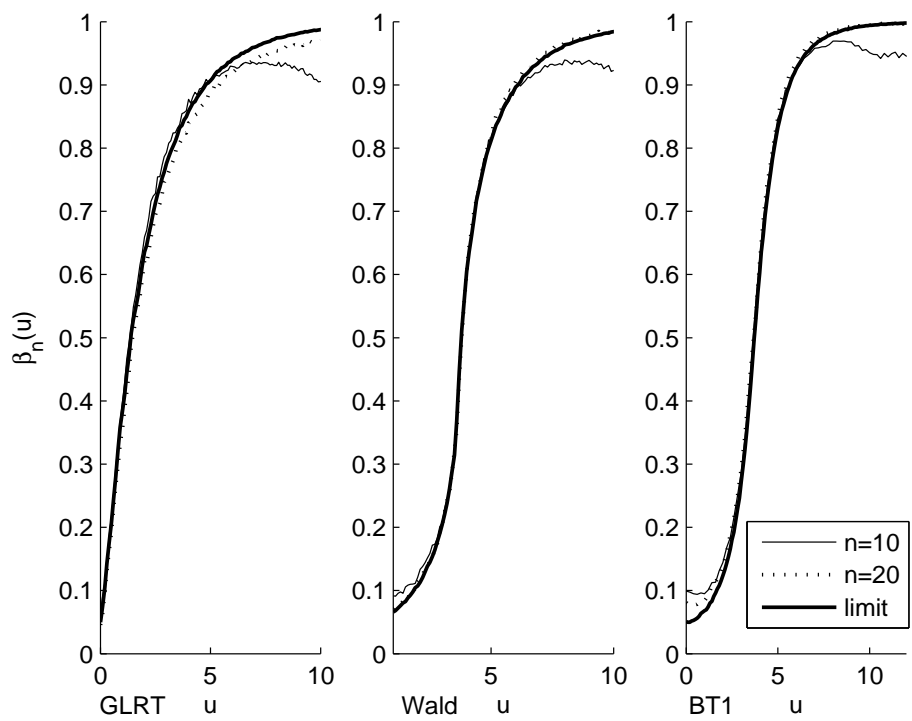}
\caption{\label{PF_dis}Power functions of GLRT, WT and BT1 in discontinuous
  case with $\rho=3$}
\end{center}
\end{figure}

\begin{figure}[htb]
\begin{center}
\includegraphics[width=\textwidth]{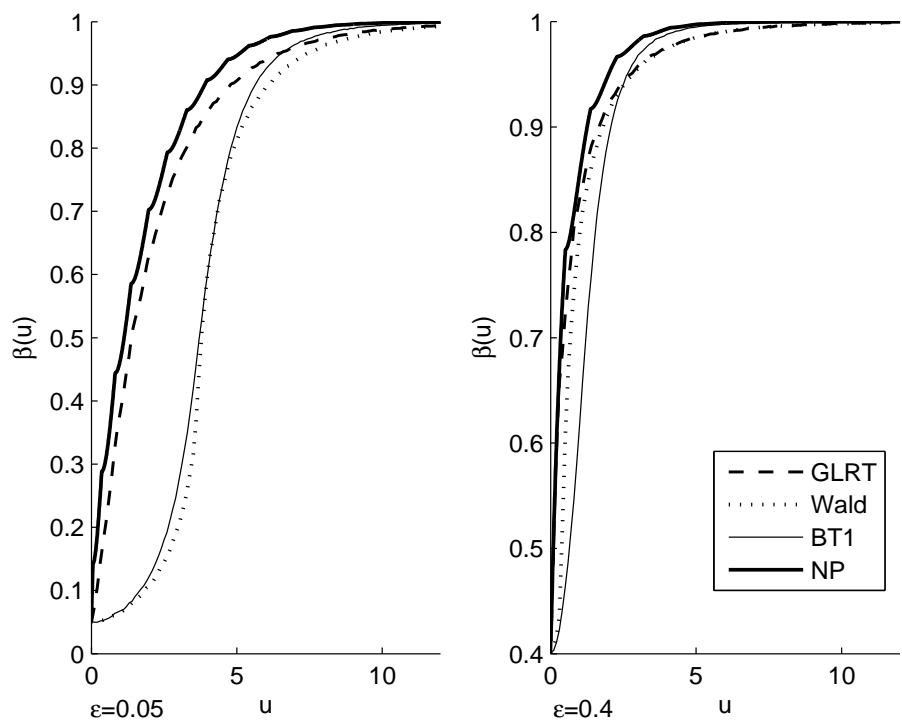}
\caption{\label{PF_dis_rho_3}Comparison of limit power functions in
  discontinuous case with $\rho=3$}
\end{center}
\end{figure}

\end{document}